\documentclass[12pt,a4paper,reqno,draft]{amsart}
\usepackage{amssymb,amsfonts}
\usepackage[utf8]{inputenc}
\usepackage[english]{babel}
\usepackage{color}
\usepackage{enumerate}

\advance\textwidth30mm \advance\hoffset-12mm
\advance\textheight20mm \advance\voffset-12mm
\newtheorem{thm}{Theorem}
\newtheorem{lem}{Lemma}

\theoremstyle{remark}
\newtheorem{rem}{Remark}
\theoremstyle{definition}

\newcommand\sign{\operatorname{sign}}
\renewcommand\Re{\operatorname{Re}}
\renewcommand\Im{\operatorname{Im}}
\newcommand\supp{\operatorname{supp}}

\newcommand\Cdot{{\mskip2mu\cdot\mskip2mu}}
\newcommand\PW{\mathit{PW}}
\newcommand\res{\operatorname*{res}}

\begin{document}
%UDC 517.5

\title{The Nikolskii constant in odd dimensions}
\author{D.\,V.~Gorbachev}
\address{Saint Petersburg State University}
\email{dvgmail@mail.ru}
\thanks{The study was carried out with the financial support of the Ministry of Science
and Higher Education of the Russian Federation in the framework of a scientific
project under agreement No.~075-15-2025-013.}
%The study was carried out with the financial support of the Ministry of Science
%and Higher Education of the Russian Federation in the framework of a scientific
%project under agreement No.~075-15-2025-013.
\date{}

\begin{abstract}
Let $\varphi_{d}(|\Cdot|)$ be the radial extremal function in the problem for
the sharp Nikolskii constant $\mathcal C_d^{-1}=\inf \|f\|_{1}$ over functions
$f\in\PW_{1}^{1}(\mathbb R^{d})$, $f(0)=1$. For every odd dimension,
we construct an entire function $\Phi$ of exponential type $1/2$ such that
$\varphi_d(z)=\Phi(z)\Phi(-z)$, and $\Phi$ satisfies a quadratic functional
equation and a second-order linear differential equation with polynomial
coefficients. Thus, the problem of finding the extremal function is reduced
to a spectral problem with finitely many parameters. This result extends a
recent one-dimensional result, but uses a different method. For example, in
dimension $d=3$ it leads to a seven-diagonal spectral scheme that allows us to
compute $\mathcal C_3$ to high accuracy. Even dimensions remain open within
this approach.

\end{abstract}



\maketitle

\section{Introduction}

\subsection{Main result}
Consider the extremal problem for the sharp Nikolskii constant
\begin{equation}\label{c-prob}
\mathcal C_d^{-1}=\inf\{\|f\|_{1}\colon f\in\PW_{1}^{1}(\mathbb
R^{d}),\ f(0)=1\},
\end{equation}
where $\PW_{r}^{p}(\mathbb R^{d})$ is the Paley--Wiener class of
functions of spherical exponential type at most~$r$ that belong to
$L^{p}(\mathbb R^{d})$.

It was proved in \cite{Da21} that problem~\eqref{c-prob} has a unique radial
extremal function $\varphi_{d}(|\Cdot|)\in\PW_{1}^{1}(\mathbb{R}^{d})$,
which can be written as
\[
\varphi_{d}(z)=\prod_{k=1}^{\infty} \Bigl(1-\frac{z^{2}}{\tau_{k}^{2}}\Bigr),\quad
z\in \mathbb{C},
\]
where $0<\tau_1<\tau_2<\ldots$ are the positive zeros of $\varphi_{d}$ and
$\tau_{k}\sim \pi k$ as $k\to \infty$ (see Section~\ref{sec-prop-phi}).

For $d=1$, it was shown in \cite{Bo25} that the extremal function factors as
$\varphi_{1}(z)=\Phi(z)\Phi(-z)$ and that $\Phi$ satisfies both a functional
equation and a linear differential equation. Below we obtain analogous
equations in all odd dimensions.

From now on, write $\varphi=\varphi_{d}$. As in the one-dimensional case, set
\begin{equation}\label{varphi-Phi}
\varphi(z)=\Phi(z)\Phi(-z),
\end{equation}
where
\begin{equation}\label{intro-Phi}
\Phi(z)=\prod_{k=1}^{\infty}\Bigl(1+(-1)^k\frac{z}{\tau_k}\Bigr).
\end{equation}
By Lemma \ref{lem-Phi}, $\Phi$ is an entire function of finite exponential
type.

Let $V_d$ be the volume of the unit ball $B_{1}^{d}\subset \mathbb{R}^{d}$ and
\begin{equation}\label{def-ad}
a_d=\frac{1}{2V_d\mathcal C_d}.
\end{equation}

\begin{thm}\label{thm-main}
Let $d=2m+1$, $m\in\mathbb Z_{\ge 0}$, $z\in\mathbb C$.

\textup{(a)} There exists an even polynomial $Q_d$ such that the quadratic
functional identity
\begin{equation}\label{intro-FE}
z^{d+1}\bigl(\Phi'(z)\Phi(-z)+\Phi'(-z)\Phi(z)\bigr)
-2R_d(z)\Phi(z)\Phi(-z)=-2a_d,
\end{equation}
holds, where $R_d(z)=a_d+z^2Q_d(z)$. Moreover, $Q_1=0$ and
$\deg Q_d\le d-3$ for $d\ge3$.

\textup{(b)} There exists an even polynomial $K_d$ such that $\Phi$ satisfies
the linear differential equation
\begin{equation}\label{intro-ODE}
z^{d+1}\Phi''(z)
+\bigl((d+1)z^d-2R_d(z)\bigr)\Phi'(z)
+\bigl(K_d(z)-R_d'(z)\bigr)\Phi(z)=0.
\end{equation}
The polynomial $K_d$ has degree $d+1$ and leading coefficient $1/4$.

\textup{(c)} The function $\Phi$ has exact exponential type $1/2$ and
\[
\Phi(x)=O\bigl(|x|^{-(d+1)/2}\bigr),\quad |x|\to\infty.
\]
\end{thm}

The polynomials $Q_d$ and $K_d$ are determined by the extremal function
$\varphi$. Set
\[
\sigma(x)=\sign\varphi(x),\quad x\in\mathbb R.
\]
More precisely, $Q_d$ is given by \eqref{def-Qd}, where $p_n$ are the
coefficients of the even polynomial~$P_d$ of degree at most $d-1$ defined by
\begin{equation}\label{sigma-P}
\widehat\sigma(t)=\frac12\,P_d(t),\quad t\in(-1,1).
\end{equation}
This identity follows from the extremality condition \eqref{extremality-1D}
and the oddness of the dimension $d=2m+1$: since $|x|^{d-1}=x^{2m}$, for the
one-dimensional Fourier transform, in the sense of distributions,
\[
\widehat{x^{2m}\sigma}
=(-1)^m\partial_t^{2m}\widehat\sigma.
\]
The polynomial $K_d$ is defined by \eqref{def-Kd}, while formula
\eqref{Kd-formula} expresses it through the coefficients of $R_d$ and a finite
number of coefficients of the logarithmic derivative $\Phi'/\Phi$.

Thus, the coefficients of $Q_d$ and $K_d$ are internal parameters of the
extremal problem and are determined together with $\Phi$. As a result, the
original nonlinear extremal problem is reduced to a spectral problem for a
second-order linear differential equation with polynomial coefficients. In particular, in dimension $d=3$ it leads to a seven-diagonal spectral scheme that allows us to compute $\mathcal C_3$ and the coefficients of the differential equation for $\Phi$ with high accuracy.

\subsection{Outline of the proof}
To prove \textup{(a)}, the interlacing of the zeros of $\Phi(z)$ and
$\Phi(-z)$ allows us to construct the analytic function $S$,
\[
S(z)=\frac{4}{\pi}\arctan B(z),\quad
B(z)=\frac{\Phi(z)+i\Phi(-z)}{\Phi(-z)+i\Phi(z)},
\]
where $B$ is a meromorphic inner function in the upper half-plane
$\Im z>0$. Moreover, $\Re S=\sigma$ on $\mathbb{R}$ away from the zeros of $\varphi$. The
form of $S$ was suggested by \cite{Ge38}. The function $S$ has one-sided
spectrum, and $\widehat S=2\widehat\sigma$ on $(0,\infty)$. Hence,
\eqref{sigma-P} gives $\widehat S=P_d$ on $(0,1)$. The coefficients of $P_d$
determine the polynomial $Q_d$, which allows us to construct from
$A(z)=(\pi i/4)S'(z)$ the meromorphic function
\[
\Theta_d(z)=z^dA(z)-zQ_d(z)-\frac{a_d}{z}
\]
with a spectral gap. We then prove that the defect
$H_d=\varphi\Theta_d+a_d/z$ is entire and has Fourier transform supported at a
single point, and hence is a polynomial. Evaluating~$H_d$ at the zeros of
$\varphi$ and applying the Plancherel--Polya inequality, we obtain
$H_d\equiv0$. This yields the functional equation \eqref{intro-FE}.

To prove \textup{(b)}, we differentiate \eqref{intro-FE} at the zeros of
$\Phi$. It follows that
\[
\mathcal P_d(z)=- \frac{z^{d+1}\Phi''(z)+((d+1)z^d-2R_d(z))\Phi'(z)}{\Phi(z)}
\]
has no poles. Lemma~\ref{lem-logder} shows that $\mathcal P_d$ is a polynomial.
After passing to the modified function
$\widetilde{\Phi}=e^{-h_d}\Phi$, where
$h_d'(z)=R_d(z)/z^{d+1}$, we obtain
\[
\widetilde{\Phi}''+\frac{d+1}{z}\,\widetilde{\Phi}'+q_d(z)\widetilde{\Phi}=0,
\]
where $q_d$ is even. This implies that the polynomial
$K_d=\mathcal P_d+R_d'$ is even. The Liouville transform reduces this equation
on the positive half-line to $u''+Vu=0$, where the leading behavior of the
potential is determined by $K_d$. The density of the positive zeros of $\Phi$
is $1/(2\pi)$ because $\tau_{2n-1}\sim2\pi n$. Upper and lower Sturm-type
bounds for the number of zeros force $\deg K_d=d+1$ and show that its leading
coefficient is $1/4$. It is important that only the asymptotic relation
$\tau_k\sim\pi k$ is used here (cf.~\cite{Bo25}). The remaining
coefficients of $K_d$ are determined by \eqref{Kd-formula}.

Finally, to prove \textup{(c)}, the above equation and the properties of $K_d$
give $V(x)=1/4+O(x^{-2})$ as $x\to\infty$. The energy estimate \eqref{energy}
gives $u(x),u'(x)=O(1)$, and therefore
$\Phi(x)=O\bigl(|x|^{-(d+1)/2}\bigr)$ as $|x|\to\infty$. Since $\Phi$ has
finite exponential type, $\widehat\Phi$ has compact support. Taking the
Fourier transform of \eqref{intro-ODE} and using uniqueness of solutions away
from the zeros of the coefficient $1/4-t^2$, we obtain
$\supp\widehat\Phi\subset[-1/2,1/2]$. Hence the type of $\Phi$ is at most
$1/2$, while $\varphi(z)=\Phi(z)\Phi(-z)$ and the exact type $1$ of $\varphi$
show that it is exactly $1/2$.

The proof is substantially different from the one-dimensional case. For
$d=1$, the polynomial $P_1$ has degree~$0$, so $Q_1$ is absent. Starting with
$d=3$, additional low-frequency coefficients appear in $Q_d$, together with
the corresponding regularization. At the same time, the derivation of the
linear differential equation is uniform for all odd~$d$: first one proves
that $\mathcal P_d$ is entire and polynomial, and then determines its leading
coefficient from the zero density. For $d=1$, this argument gives a new proof
of the differential equation obtained in \cite{Bo25}.

\subsection{Known estimates for the constant}
Following \cite{Da21}, introduce the normalized Nikolskii constant
\begin{equation}\label{def-Lstar}
\mathcal L^*(d)=\frac{(2\pi)^d}{V_d}\,\mathcal C_d.
\end{equation}
It was proved in \cite{Da21} that for all $d\ge1$
\begin{equation}\label{Da21-bounds}
2^{-d}\le \mathcal L^*(d)\le {}_1F_2\Bigl(\frac d2;\frac d2+1,\frac d2+1;
-\frac{\beta_d^2}{4}\Bigr),
\end{equation}
where $\beta_d$ is the first positive zero of the Bessel function $J_{d/2}$.
As a consequence,
\[
\mathcal L^*(d)\le \bigl(\sqrt{2/e}\,\bigr)^d \bigl(1+O(d^{-2/3})\bigr),\quad
d\to\infty,
\]
where $\sqrt{2/e}=0.857\ldots$.

For example, for $d=1$ we have $\mathcal L^*(1)=\pi\mathcal C_1$ and
\[
\frac12\le \mathcal L^*(1)\le \frac1\pi\int_0^\pi\frac{\sin t}{t}\,dt=
0.589\ldots.
\]
The numerical value obtained in \cite{Bo25} is
\[
\mathcal L^*(1)
=0.54092882190183058939\ldots.
\]

\subsection{Organization of the paper}
Section~\ref{sec-prop-phi} collects and proves the properties of the extremal
function $\varphi$ needed below. Section~\ref{sec-utils} contains the auxiliary
results used later. Sections~\ref{sec-functional}--\ref{sec-Phi} prove,
respectively, parts~\textup{(a)}--\textup{(c)} of Theorem~\ref{thm-main}. In
Section~\ref{sec-d1}, the resulting formulas are compared with the known
one-dimensional result. Section~\ref{sec-d3} is devoted to the case $d=3$: we
derive a spectral boundary-value problem, construct a seven-diagonal numerical
scheme, and compute $\mathcal C_3$ to high accuracy. Section~\ref{sec-final}
contains further consequences and interpretations of the main identities,
including an equation for the extremal function itself, the finite part of a
series over the zeros of $\varphi$, and a zeta interpretation of the
coefficients of the polynomial $Q_d$. The obstruction to extending the method
to even dimensions is also discussed.

\section{Properties of the extremal function $\varphi$}\label{sec-prop-phi}

In this section we collect the main properties of the extremal function that
will be used later. Some of these facts are known \cite[Th.~4.2]{Da21}, but
for convenience we include short proofs.

\subsection{Existence of a real extremal function}
Let
\[
(f_n)_{n=1}^{\infty}\subset \PW_{1}^{1}(\mathbb R^d),\quad
f_n(0)=1,\quad \|f_n\|_1\to \mathcal C_d^{-1}.
\]
By the compactness theorem \cite[Th.~3.3.6]{Ni75}, there is a subsequence
$(f_{n_j})$ converging locally uniformly to a function
$f\in \PW_{1}^{1}(\mathbb R^d)$. Then $f(0)=1$ and, for every $R>0$,
\[
\int_{|x|\le R}|f(x)|\,dx= \lim_{j\to\infty} \int_{|x|\le R}|f_{n_j}(x)|\,dx\le
\mathcal C_d^{-1}.
\]
Letting $R\to\infty$, we obtain $\|f\|_1\le \mathcal C_d^{-1}$. Thus, $f$ is
an extremal function.

Replacing $f(z)$, if necessary, by $(1/2)(f(z)+f^{\#}(z))$, where
\[
f^{\#}(z)=\overline{f(\overline z)},
\]
we obtain a function with no larger $L^{1}$ norm. Thus, we may assume that $f$
is real-valued on $\mathbb R^d$.

\subsection{Reduction to a radial function}
Average the extremal function by setting
\[
\varphi(z)= \int_{|\xi|=1} f(z\xi)\,d\nu_{d}(\xi),\quad z\in\mathbb C,
\]
where $\nu_{d}$ is the normalized surface measure on $\mathbb{S}^{d-1}$ for
$d>1$ and $\nu_{1}=(1/2)(\delta_{-1}+\delta_{1})$. The function $\varphi$ is an
even real entire function, $\varphi(|\Cdot|)\in
\PW_{1}^{1}(\mathbb R^d)$, and $\varphi(0)=1$. By the triangle
inequality, Fubini's theorem, and the polar-coordinate formula,
$\|\varphi(|\Cdot|)\|_1\le \|f\|_1$. Hence $\varphi(|\Cdot|)$ is also an
extremal function.

\subsection{Boundedness and exact exponential type}
By Nikolskii's inequality,
\[
\|\varphi\|_\infty\le \mathcal C_d\|\varphi(|\Cdot|)\|_1=1.
\]

We show that the exponential type of $\varphi$ is exactly $1$. If it were less
than~$1$, one could choose $\rho<1$ so that the function
$x\mapsto\varphi(|x|/\rho)$ had type at most~$1$. It takes the value~$1$ at the
origin and
\[
\int_{\mathbb R^d}|\varphi(|x|/\rho)|\,dx= \rho^d\int_{\mathbb
R^d}|\varphi(|x|)|\,dx
<
\mathcal C_d^{-1},
\]
which contradicts extremality.

\subsection{Extremality conditions}
Let $\mathcal V$ be the subspace of radial functions
$h(|\Cdot|)\in\PW_{1}^{1}(\mathbb R^d)$ such that $h(0)=0$. The
extremality of $\varphi$ is equivalent to
\begin{equation}\label{phi-ineq-cond}
\|\varphi(|\Cdot|)\|_1\le \|\varphi(|\Cdot|)+h(|\Cdot|)\|_1, \quad h\in\mathcal
V.
\end{equation}
Since $\varphi$ is a nonzero entire function, the set
$\{\xi\in\mathbb R^d\colon \varphi(|\xi|)=0\}$ has measure zero. Therefore, by
Theorem~4.2.2 in \cite{Sh71}, inequality \eqref{phi-ineq-cond} is equivalent
to
\begin{equation}\label{phi-h}
\int_{\mathbb R^d} \sigma(|\xi|)h(|\xi|)\,d\xi= 0,\quad h\in\mathcal V,\quad
\sigma=\sign\varphi.
\end{equation}

Now let $g(|\Cdot|)\in\PW_{1}^{1}(\mathbb R^d)$ be arbitrary. Then
\[
g(|\Cdot|)-g(0)\varphi(|\Cdot|)\in\mathcal V,
\]
and therefore, by \eqref{phi-h},
\[
0= \int_{\mathbb R^d} \sigma(|\xi|)
\bigl(g(|\xi|)-g(0)\varphi(|\xi|)\bigr)\,d\xi.
\]
Since
\[
\int_{\mathbb R^d}|\varphi(|\xi|)|\,d\xi=
\mathcal C_d^{-1},
\]
we obtain the extremality condition in the form
\begin{equation}\label{g0-phi}
g(0)= \mathcal C_d \int_{\mathbb R^d} \sigma(|\xi|)g(|\xi|)\,d\xi,\quad
g(|\Cdot|)\in\PW_{1}^{1}(\mathbb R^d).
\end{equation}

Passing to polar coordinates and writing $\omega_{d-1}=dV_{d}$ for the area
of the unit sphere $\mathbb S^{d-1}$, we obtain the following one-dimensional
form. If $g$ is an even entire function of type at most~$1$ such that
$|x|^{d-1}g(x)\in L^1(\mathbb R)$, then
\begin{equation}\label{extremality-1D}
g(0)=\frac{\mathcal C_d\omega_{d-1}}{2} \int_{\mathbb
R}\sigma(x)g(x)|x|^{d-1}\,dx.
\end{equation}

\subsection{Zeros of the extremal function}
The function $\varphi$ has only simple real zeros $\pm \tau_{k}$,
$k\in\mathbb{N}$. Indeed, if $\varphi$ had a nonreal zero or a real zero of
multiplicity greater than one, then it could be written as $\varphi=pg$, where
$p$ is a nonconstant even real polynomial that is nonnegative on $\mathbb R$.
Then, for every $t>0$,
\[
p(t|\Cdot|)g(|\Cdot|)\in \PW_{1}^{1}(\mathbb R^d),
\]
because the ratio $p(t|x|)/p(|x|)$ is bounded for all sufficiently large
$|x|$, and multiplication by a polynomial does not increase the exponential
type. Moreover,
\[
\sign\varphi(x)= \sign\bigl(p(tx)g(x)\bigr)= \sign g(x)
\]
for almost every $x\in\mathbb R$. Therefore, by \eqref{g0-phi},
\[
1= \mathcal C_d \int_{\mathbb R^d} p(t|\xi|)|g(|\xi|)|\,d\xi,\quad t>0.
\]
The right-hand side is a nonconstant polynomial in $t$, which is impossible.

\subsection{Canonical product and zero asymptotics}
Since $\varphi\in\PW_{1}^{\infty}(\mathbb R)$, it belongs to the
Cartwright class. Since $\varphi$ is even and all its zeros are
simple and real, we obtain the canonical representation
\[
\varphi(z)= \prod_{k=1}^{\infty} \Bigl(1-\frac{z^2}{\tau_k^2}\Bigr),\quad
z\in\mathbb C.
\]

For a function $f$ with discrete positive zeros, let
\[
n_f(X)=|\{x\in(0,X]\colon f(x)=0\}|,\quad X>0,
\]
denote the number of its distinct zeros in $(0,X]$. By the zero-density
theorem for Cartwright-class functions \cite[Ch.~V, Th.~11]{Le80} and
the evenness of $\varphi$,
\[
n_\varphi(X)=\frac{X}{\pi}+o(X),\quad X\to\infty.
\]
Hence,
\begin{equation}\label{tau-props}
\tau_k\sim\pi k,\quad k\to\infty.
\end{equation}

\subsection{Uniqueness of the radial extremal function}
Let $\psi(|\Cdot|)$ be another radial extremal function. Then, by
\eqref{g0-phi},
\[
1= \mathcal C_d \int_{\mathbb R^d}|\psi(|\xi|)|\,d\xi= \mathcal C_d
\int_{\mathbb R^d} \sigma(|\xi|)\psi(|\xi|)\,d\xi,
\]
so $|\psi|= \psi \sign\varphi$ almost everywhere on $\mathbb R$. Similarly,
interchanging $\varphi$ and $\psi$, we obtain
$|\varphi|= \varphi \sign\psi$ almost everywhere. By these two identities
and continuity, $\varphi\psi=|\varphi\psi|$ everywhere on $\mathbb R$. Since all zeros of $\varphi$ and $\psi$ are simple, both functions change sign
at each of their zeros. Hence their zero sets coincide, and the canonical products give
$\psi\equiv\varphi$.

\subsection{Spectral characterization of the zeros}\label{subsec-zero-characterization}
Throughout this subsection, $\mathcal F_d$ denotes the $d$-dimensional Fourier
transform
\[
\mathcal F_d (f)(\xi)
=\int_{\mathbb R^d}f(x)e^{-i\langle x,\xi\rangle}\,dx,
\]
which is extended to $\mathcal S'(\mathbb R^d)$ by duality.
We shall use the normalized Bessel function
\[
j_\alpha(z)=2^\alpha\Gamma(\alpha+1)\,\frac{J_\alpha(z)}{z^\alpha},
\quad j_\alpha(0)=1.
\]
Let $\mathbf 1_E$ denote the characteristic function of a set $E$. 
Then (see, e.g., \cite{Da21})
\[
\mathcal F_d(\mathbf 1_{B_r^d})(\xi)
=V_dr^d j_{d/2}(r|\xi|).
\]

Since the sign of $\varphi$ changes at each zero $\tau_n$, we have
\[
\sigma(|\Cdot|)=2\sum_{n=1}^\infty(-1)^{n-1}\mathbf 1_{B_{\tau_n}^d},
\]
where the series is understood in the Abel sense: the limit as $\rho\uparrow1$
after multiplying the $n$th term by $\rho^{n-1}$. The validity of this
passage follows directly from the nesting of the balls: the corresponding
Abel sums are uniformly bounded and converge locally almost everywhere to
$\sigma(|\Cdot|)$.

On the other hand, the extremality condition is equivalent to
\[
\mathcal F_d(\sigma(|\Cdot|))
=\frac1{\mathcal C_d}\quad\text{on $B_1^d$}.
\]
Since $a_d=(2V_d\mathcal C_d)^{-1}$, we obtain the spectral identity
\begin{equation}\label{zero-characterization}
\sum_{n=1}^\infty(-1)^{n-1}\tau_n^d
j_{d/2}(\tau_nu)=a_d,\quad |u|<1,
\end{equation}
where the series is again understood in the Abel sense.

The converse statement also holds. Let $a>0$, and let the real-valued function
$\psi(|\Cdot|)\in \PW_1^1(\mathbb R^d)$ satisfy $\psi(0)=1$ and have only
simple real zeros. Denote its positive zeros by $0<t_1<t_2<\ldots.$ If
\[
\sum_{n=1}^\infty(-1)^{n-1}t_n^d
j_{d/2}(t_nu)=a,\quad |u|<1,
\]
in the Abel sense, then
\[
\mathcal F_d(\sign\psi(|\Cdot|))=2V_da
\quad\text{on $B_1^d$}.
\]
Set $\mathcal C=(2V_da)^{-1}$. From the last identity, first for functions of
exponential type less than~$1$ and then, by dilation, for functions of type at
most~$1$, we obtain
\[
g(0)=\mathcal C\int_{\mathbb R^d}\sign\psi(|\xi|)g(|\xi|)\,d\xi,
\quad
g(|\Cdot|)\in \PW_1^1(\mathbb R^d).
\]
Hence $\mathcal C_d\le \mathcal C$. On the other hand, taking $g=\psi$ gives
\[
1=\mathcal C\|\psi(|\Cdot|)\|_1.
\]
By the definition of $\mathcal C_d$,
$\|\psi(|\Cdot|)\|_1\ge\mathcal C_d^{-1}$, and therefore $\mathcal C\le\mathcal C_d$.
Thus,
\[
\mathcal C=\mathcal C_d,\quad a=a_d,
\]
and $\psi$ is an extremal function. By uniqueness, $\psi=\varphi$.

For $d=1$,
\[
j_{1/2}(z)=\frac{\sin z}{z},
\]
and after passing from our normalization of exponential type~$1$ to the
normalization of type~$\pi$, relation \eqref{zero-characterization} becomes
formula~(2.6) in \cite{Bo25}.

\section{Auxiliary results}\label{sec-utils}

This section collects standard facts used in the main proofs. The letter $C$
denotes positive constants whose value may change from line to line.

\begin{lem}\label{lem-Phi}
The function $\Phi$ defined by \eqref{intro-Phi} is entire and of finite
exponential type.
\end{lem}

\begin{proof}
Group neighboring factors:
\[
\Phi(z)=\prod_{k=1}^\infty \Bigl(1-\frac{z}{\tau_{2k-1}}\Bigr)
\Bigl(1+\frac{z}{\tau_{2k}}\Bigr).
\]
Since $1/\tau_{k}\downarrow 0$,
\[
C_{1}=\sum_{k=1}^{\infty} \Bigl(\frac1{\tau_{2k-1}}-\frac1{\tau_{2k}}\Bigr)<\infty.
\]
Also, \eqref{tau-props} implies that $\tau_{2k-1}\tau_{2k}\ge Ck^2$ for all
sufficiently large $k$. Hence the product converges locally uniformly. To
estimate the growth, use
\[
\Bigl|\Bigl(1-\frac{z}{\tau_{2k-1}}\Bigr)
\Bigl(1+\frac{z}{\tau_{2k}}\Bigr)\Bigr|\le
\exp\Bigl(|z|\Bigl(\frac1{\tau_{2k-1}}-\frac1{\tau_{2k}}\Bigr)\Bigr)
\Bigl(1+\frac{|z|^2}{\tau_{2k-1}\tau_{2k}}\Bigr).
\]
Therefore, for $r=|z|$,
\[
|\Phi(z)|\le e^{C_{1}r}\prod_{k=1}^{\infty}\Bigl(1+\frac{C_{2}^{2}
r^2}{k^2}\Bigr) =e^{C_{1}r}\,\frac{\sinh(\pi C_{2}r)}{\pi C_{2}r}\le
C_{3}e^{C_{4}r}.
\]
\end{proof}

Following \cite[Ch.~VII]{Le80}, an entire function $E$ is called a
Hermite--Biehler function if it has no zeros in the half-plane $\Im z\le 0$
and
\begin{equation}\label{E-E}
|E(z)|<|E^\#(z)|,\quad \Im z>0.
\end{equation}

\begin{lem}\label{lem-HB}
The function
\[
E(z)=\Phi(z)+i\Phi(-z)
\]
is a Hermite--Biehler function.
\end{lem}

\begin{proof}
We use the Hermite--Biehler theorem \cite[Ch.~VII, Th.~$3'$]{Le80}.
Set $P(z)=\Phi(z)$ and $Q(z)=\Phi(-z)$, so that $E=P+iQ$. The zeros of $P$ and
$Q$ are, respectively,
\begin{align*}
\ldots<-\tau_6<-\tau_4<-\tau_2<&\tau_1<\tau_3<\tau_5<\ldots,\\
\ldots<-\tau_5<-\tau_3<-\tau_1<&\tau_2<\tau_4<\tau_6<\ldots.
\end{align*}
Thus they are simple, real, and strictly interlacing. Moreover, the canonical
product \eqref{intro-Phi}, and the analogous product for $\Phi(-z)$, contain no
exponential factors, and the constants in front of them are both $1$, hence
have the same sign. All assumptions of the Hermite--Biehler theorem are
satisfied. The direction of inequality \eqref{E-E} can also be checked at
$z=iy$, $y\downarrow0$, from the Taylor expansion
\[
|E(iy)|^2-|E^\#(iy)|^2=8cy+O(y^2),
\]
where
\[
c=\Phi'(0)=\sum_{k\ge1}\frac{(-1)^k}{\tau_k}<0
\]
by the Leibniz criterion.
\end{proof}

Next, let $\mathcal S(\mathbb R)$ be the Schwartz space,
$\mathcal S'(\mathbb R)$ the space of tempered distributions, and
$\mathrm{BMO}(\mathbb R)\subset \mathcal S'(\mathbb R)$ the space of
functions of bounded mean oscillation. We use the Fourier transform
\[
\widehat f(t)=\int_{\mathbb R}f(x)e^{-itx}\,dx,
\]
extended to $\mathcal S'(\mathbb R)$ by duality. In particular,
$\widehat{1}=2\pi \delta_{0}$. The symbol $*$ denotes convolution. We write
$\mathcal H$ for the Hilbert transform and use its standard extension
\[
\mathcal H\colon L^\infty(\mathbb R)\to\mathrm{BMO}(\mathbb R),
\]
for which
\[
\widehat{\mathcal H f}(t)=-i\sign t\,\widehat f(t),\quad t\ne0.
\]
All boundary values below are understood in the sense of distributions. For a
function~$F$ with boundary values on the real axis, write
\[
F^\pm(x)=\lim_{\varepsilon\downarrow0}F(x\pm i\varepsilon).
\]

\begin{lem}\label{lem-analytic-signal}
Let $F$ be analytic in the upper half-plane and suppose that $\Re F$ is
bounded there. If $u$ is the boundary value of $\Re F$, then the boundary
value of $F$ exists in $\mathrm{BMO}(\mathbb R)$ and
\begin{equation}\label{F-u}
\supp\widehat F\subset[0,\infty),\quad
\widehat F=2\widehat u\quad\text{on $(0,\infty)$}.
\end{equation}
Moreover, a possible component of $\widehat F$ supported at the origin has the
form $\gamma\delta_0$, where $\gamma\in \mathbb{C}$.
\end{lem}

\begin{proof}
We use standard facts about boundary values of analytic functions and the
Hilbert transform \cite[Ch.~VI, X]{Ko98}. Since $\Re F$ is bounded, we
have $u\in L^\infty(\mathbb R)$ and $\mathcal Hu\in\mathrm{BMO}(\mathbb R)$.
The harmonic function $\Im F$ is harmonic conjugate to $\Re F$, so on the
boundary
\[
F=u+i\mathcal Hu+ic
\]
for some constant $c\in\mathbb R$. Hence, on $\mathbb R\setminus\{0\}$,
\[
\widehat F(t)=(1+\sign t)\widehat u(t),
\]
which gives \eqref{F-u}.
Any distribution supported at the origin is a finite linear combination of
$\delta_0$ and its derivatives, and its inverse Fourier transform is a
polynomial. Since $F\in\mathrm{BMO}(\mathbb R)$ and the only polynomials in
$\mathrm{BMO}(\mathbb R)$ are constants, only a term of the form
$\gamma\delta_0$, $\gamma\in\mathbb C$, may occur at the origin.
\end{proof}

\begin{lem}\label{lem-multiplier}
Let $f\in\PW_{1}^{\infty}(\mathbb R)$ and
$T\in\mathcal S'(\mathbb R)$. Then
$fT\in \mathcal S'(\mathbb R)$ and
\begin{equation}\label{f-T}
\widehat{fT}=\frac1{2\pi}\,\widehat f*\widehat T,\quad
\supp\widehat{fT}\subset\supp\widehat f+\supp\widehat T.
\end{equation}
\end{lem}

\begin{proof}
By Bernstein's inequality, $\|f^{(k)}\|_\infty\le \|f\|_\infty$, $k\ge0$.
Therefore multiplication by~$f$ acts continuously on $\mathcal S(\mathbb R)$
and, by duality, on $\mathcal S'(\mathbb R)$.

By the Paley--Wiener--Schwartz theorem,
$\supp\widehat f\subset[-1,1]$ (see \cite[Th.~7.3.1]{Ho90}). Therefore
the convolution $\widehat f*\widehat T$ is defined and
$\widehat{fT}=(\widehat f*\widehat T)/(2\pi)$. Finally, the standard support
property of convolution gives the last inclusion in \eqref{f-T}.
\end{proof}

\begin{lem}\label{lem-PP}
If $g\in\PW_1^1(\mathbb R)$ and
$\{x_n\}_{n=1}^\infty\subset\mathbb R$ is uniformly separated, then
\[
\sum_{n=1}^\infty |g'(x_n)|\le C\|g\|_1.
\]
\end{lem}

\begin{proof}
Since $g'\in\PW_1^1(\mathbb R)$, it is enough to apply successively
the Plancherel--Polya inequality \cite[(1.1)]{Pe07} and Bernstein's inequality:
\[
\sum_{n=1}^\infty |g'(x_n)|\le C\|g'\|_1\le C\|g\|_1.
\]
\end{proof}

\begin{lem}\label{lem-logder}
Let $f$ be an entire function of finite order and let $U_1$, $U_2$ be
polynomials. If the function
\[
G= U_2\,\frac{f''}{f}+U_1\,\frac{f'}{f}
\]
has no poles, then it is a polynomial.
\end{lem}

\begin{proof}
If $f$ is a polynomial, the statement is obvious. Thus, assume that $f$ is
transcendental. By the logarithmic derivative estimate \cite[Cor.~2]{Gu88},
there exists a set $E\subset(1,\infty)$ of finite logarithmic measure such that,
for $|z|\notin E$, the functions $f'(z)/f(z)$ and $f''(z)/f(z)$ grow no faster
than some power of $|z|$. Hence, for the entire function $G(z)$,
\[
|G(z)|\le C|z|^N,\quad |z|\notin E,
\]
with some $C,N>0$.

Since $E$ has finite logarithmic measure, for every sufficiently large~$r$
there is an $R\notin E$ in the interval $[r,2r]$. By the maximum principle,
\[
\max_{|z|=r}|G(z)|
\le \max_{|z|=R}|G(z)|
\le CR^N
\le C2^N r^N.
\]
Thus, $G$ has polynomial growth and hence is a polynomial.
\end{proof}

\begin{lem}\label{lem-Liouville}
Let $p\in C^1(A,B)$, $q\in C(A,B)$, and let $y\in C^2(A,B)$ satisfy the
differential equation
\[
y''+py'+qy=0.
\]
Then the Liouville transform
\[
u(x)=\exp\Bigl(\frac12\int_A^x p\Bigr)y(x)
\]
reduces this equation to
\[
u''+Vu=0,\quad V=q-\frac{p'}2-\frac{p^2}{4}.
\]
\end{lem}

\begin{proof}
Direct substitution.
\end{proof}

\begin{lem}\label{lem-zero-count}
Let $V\in C[A,B]$ be real-valued, $u\in C^2[A,B]\setminus \{0\}$, and
\[
u''+Vu=0\quad \text{on $[A,B]$}.
\]

\textup{(a)} If $V\le M^2$ on $[A,B]$, where $M\ge0$, then the number of
distinct zeros of $u$ on $[A,B]$ is at most
\begin{equation}\label{a-estim}
\frac{M}{\pi}\,(B-A)+1.
\end{equation}

\textup{(b)} If $V\ge m^2$ on $[A,B]$, where $m\ge0$, then the number of
distinct zeros of $u$ on $[A,B]$ is at least
\begin{equation}\label{b-estim}
\frac{m}{\pi}\,(B-A)-1.
\end{equation}
\end{lem}

\begin{proof}
First note that all zeros of $u$ are simple, since $u(a)=u'(a)=0$ implies
$u\equiv0$ by uniqueness.

(a)~Let $a<b$ be two consecutive zeros of $u$. By Wirtinger's inequality
\cite[Sec.~7.7]{HLP52},
\[
\frac{\pi^2}{(b-a)^2}\int_a^b|u|^2\le \int_a^b|u'|^2= \int_a^bV|u|^2 \le
M^2\int_a^b|u|^2.
\]
For $M=0$ this chain is impossible, so there can be at most one zero on
$[A,B]$. If $M>0$, the distance between consecutive zeros is at least
$\pi/M$. Hence, if $N$ is the number of zeros on $[A,B]$, then
\[
(N-1)\frac{\pi}{M}\le B-A,
\]
which gives \eqref{a-estim}.

(b)~For $m=0$, estimate \eqref{b-estim} is obvious. Let $m>0$. Introduce the
Pruefer angle \cite[Ch.~4, Sec.~5]{Ze05}
\[
u=r\sin\theta,\quad u'=mr\cos\theta.
\]
Then
\[
\theta'=m\cos^2\theta+\frac Vm\,\sin^2\theta\ge m.
\]
Since $u=r\sin\theta$, zeros of $u$ correspond to values
$\theta\in\pi\mathbb Z$. As $\theta$ is increasing, the number of such values
on $[A,B]$ is at least
\[
\frac{\theta(B)-\theta(A)}{\pi}-1
\ge \frac{m}{\pi}\,(B-A)-1.
\]
\end{proof}

We use
$(x\pm i0)^{-1}$ for the boundary values of $1/x$ in
$\mathcal S'(\mathbb R)$:
\[
(x\pm i0)^{-1}
=\lim_{\varepsilon\downarrow0}\frac1{x\pm i\varepsilon}
=\operatorname{p.v.}\frac1x\mp i\pi\delta_0.
\]
We shall use the following standard Fourier-transform formulas for
distributions (see \cite[Ch.~II]{Ge64}):
\begin{equation}\label{boundary-Fourier}
\widehat{(x+i0)^{-1}} =-2\pi i\,\mathbf 1_{(0,\infty)},\quad
\widehat{(x-i0)^{-1}} =2\pi i\,\mathbf 1_{(-\infty,0)},
\end{equation}
\[
\widehat{x^r}=2\pi i^r\delta_0^{(r)},\quad r\in \mathbb Z_{\ge 0}.
\]

\begin{lem}\label{lem-truncated-power}
Let $r,N\in\mathbb Z_{\ge 0}$ and
$t_+^r=\mathbf{1}_{(0,\infty)}(t)t^r$. Then, in
$\mathcal S'(\mathbb R)$,
\[
\partial_t^N t_+^r=
\begin{cases}
\dfrac{r!}{(r-N)!}\,t_+^{r-N},&0\le N\le r,\\[3mm] r!\,\delta_0^{(N-r-1)},&N\ge
r+1.
\end{cases}
\]
\end{lem}

\begin{proof}
We have \cite[Ch.~I, Sec.~3]{Ge64}
\[
\partial_t\mathbf{1}_{(0,\infty)}=\delta_0,\quad \partial_t t_+^r=r
t_+^{r-1},\quad r\ge1.
\]
It remains to use induction.
\end{proof}

\begin{lem}\label{lem-boundary-regularity}
Let $a_0$, $a_1$, $a_2$ be polynomials and let
$h\in C[-1,1]\cap C^4(-1,1)$ satisfy
\[
(1-x^2)h^{(4)}-4xh^{(3)} +a_2(x)h''+a_1(x)h'+a_0(x)h=0, \quad -1<x<1.
\]
If
\[
\int_{-1}^1h(x)e^{itx}\,dx=O(t^{-2}), \quad |t|\to\infty,
\]
then
\[
h(-1)=h(1)=0,
\]
and $h$ extends to an entire function.
\end{lem}

\begin{proof}
At the points $\pm1$, the equation has regular singularities with indicial
polynomial, up to a nonzero factor,
\[
r(r-1)^2(r-2).
\]
For example, set $s=1-x$ and apply the Frobenius method (see, e.g.,
\cite[Ch.~4]{CL55}). In a neighborhood of $x=1$ there are three linearly
independent analytic solutions, while the fourth has the form
\[
s\log s\,v(s)+u(s),\quad v(0)\ne0,
\]
where $u,v$ are analytic. The same holds near $x=-1$. Since a term
$s\log s\,v(s)$ with $v(0)\ne0$ gives a leading Fourier-transform term of
order $\log|t|/t^2$, as $t\to \infty$ we have
\[
\widehat h(t)= \frac{h(1)e^{it}-h(-1)e^{-it}}{it} +
\frac{\log|t|}{t^2}\,\bigl(C_+e^{it}+C_-e^{-it}\bigr) +O(t^{-2})
\]
with some constants $C_\pm$. Applying the condition
$\widehat h(t)=O(t^{-2})$ along the sequences $t=\pi n$ and
$t=\pi n+\pi/2$, $n\to\infty$, for which respectively
\[
e^{it}=e^{-it}=(-1)^n,\quad e^{it}=-e^{-it}=i(-1)^n,
\]
we obtain $h(1)=h(-1)=0$. Applying the same two sequences to the logarithmic
term gives $C_+=C_-=0$.

Thus the logarithmic terms are absent and $h$ is analytic at $x=\pm1$. Away
from these two points, every point of the complex plane is an ordinary point
of the equation because $a_0$, $a_1$, $a_2$ are polynomials. By the standard
theory of linear differential equations with analytic coefficients
\cite[Ch.~3, Sec.~7]{CL55}, together with the removability of the
singularities at $\pm1$ already established, $h$ extends to an entire
function.
\end{proof}

\begin{lem}\label{lem-periodic-averaging}
Let $f\in L^\infty(\mathbb R)$ be $2\pi$-periodic,
$\eta\in\mathcal S(\mathbb R)$, and $\widehat\eta(0)=1$. If
\[
\eta_R(x)=R^{-1}\eta(x/R),\quad R>0, 
\]
then
\[
\int_{\mathbb R}f(x)\eta_R(x)\,dx \to \langle
f\rangle=\frac1{2\pi}\int_0^{2\pi}f(x)\,dx, \quad R\to\infty.
\]
\end{lem}

\begin{proof}
The function $f-\langle f\rangle$ has mean zero over a period, so its
primitive
\[
F(x)=\int_0^x\bigl(f(t)-\langle f\rangle\bigr)\,dt
\]
is bounded and $2\pi$-periodic. Integrating by parts, we get
\[
\int_{\mathbb R}(f(x)-\langle f\rangle)\eta_R(x)\,dx
=-\int_{\mathbb R}F(x)\eta_R'(x)\,dx
=O(R^{-1}),
\]
since $\|\eta_R'\|_1=R^{-1}\|\eta'\|_1$. As
$\int_{\mathbb R}\eta_R=\widehat\eta(0)=1$, the result follows.
\end{proof}

\section{Proof of Theorem \ref{thm-main}~(a)}\label{sec-functional}

We split the proof into short steps.

\subsection{The inner function $B$}
Define
\[
B(z)=\frac{\Phi(z)+i\Phi(-z)}{\Phi(-z)+i\Phi(z)}.
\]
By Lemma~\ref{lem-HB},
\[
|B(z)|<1,\quad \Im z>0,\quad |B(x)|=1,\quad x\in\mathbb R,
\]
so $B$ is a meromorphic inner function in the upper half-plane (see, e.g.,
\cite{Su25}).

Also, a direct computation using \eqref{varphi-Phi} gives
\begin{equation}\label{Re-B}
\Re B(x)=\frac{2\varphi(x)}{\Phi(x)^2+\Phi(-x)^2}.
\end{equation}

\subsection{The analytic signature $S$}
Set
\[
S(z)=\frac4\pi\,\arctan B(z),\quad \Im z>0.
\]
Since $|B|<1$, we use the analytic branch of the arctangent in the unit disk.
From
\[
\arctan w=\frac1{2i}\log\frac{1+iw}{1-iw},
\quad
\frac{1+iB(z)}{1-iB(z)}=\frac{i\Phi(z)}{\Phi(-z)}
\]
we obtain
\begin{equation}\label{S-log}
S(z)=\frac{2}{\pi i}\log\frac{i\Phi(z)}{\Phi(-z)}.
\end{equation}
The map $w\mapsto(4/\pi)\arctan w$ sends the unit disk onto the strip
$|\!\Re \zeta|<1$. If $|w|=1$ and $\Re w\ne0$, then
\[
\Re\frac4\pi\,\arctan w=\sign\Re w.
\]
Hence, on the real axis away from the zeros of $\varphi$, by \eqref{Re-B},
\begin{equation}\label{Re-S}
\Re S(x)=\sigma(x).
\end{equation}

Applying Lemma~\ref{lem-analytic-signal} to $F=S$ and using
\eqref{Re-S}, we obtain in $\mathcal S'(\mathbb R)$
\begin{equation}\label{S-spectrum}
\supp\widehat S\subset[0,\infty),\quad
\widehat S=2\widehat\sigma\quad\text{on }(0,\infty).
\end{equation}
At $0$, only a term $\gamma\delta_0$, $\gamma\in \mathbb{C}$, may occur.

\subsection{The polynomial $P_d$}
Now use the oddness of the dimension $d=2m+1$. Let
$\psi\in C_c^\infty((-1,1))$ be even and
\[
g(x)=\frac1{2\pi}\int_{\mathbb R}\psi(t)e^{ixt}\,dt.
\]
Then $g$ is even, has type less than~$1$, and decays rapidly. Also,
\[
g(0)=\frac1{2\pi}\int_{\mathbb R}\psi(t)\,dt,\quad \int_{\mathbb
R}x^{2m}\sigma(x)g(x)\,dx=
\frac1{2\pi}\,\bigl\langle\widehat{x^{2m}\sigma},\psi\bigr\rangle.
\]
Substituting $g$ into \eqref{extremality-1D} and using
$\omega_{d-1}=dV_d$ and $a_d=(2V_d\mathcal C_d)^{-1}$, we get
\[
\bigl\langle\widehat{x^{2m}\sigma},\psi\bigr\rangle=
\frac{4a_d}{d}\int_{\mathbb R}\psi(t)\,dt.
\]
The distribution $\widehat{x^{2m}\sigma}$ is even, so it is enough to test
against even functions, and
\begin{equation}\label{weighted-signature}
\widehat{x^{2m}\sigma}=\frac{4a_d}{d}\quad \text{on $(-1,1)$}.
\end{equation}
Thus, from
\[
\widehat{x^{2m}\sigma} =(-1)^m\,\partial_t^{2m}\widehat\sigma
\]
and the evenness of $\sigma$, it follows that on $(-1,1)$ the distribution
$2\widehat\sigma$ is an even polynomial of degree at most $d-1=2m$:
\begin{equation}\label{def-Pd}
2\widehat\sigma(t)=P_d(t)=\sum_{n=0}^m p_nt^{2n}.
\end{equation}
Moreover, \eqref{weighted-signature} gives
\begin{equation}\label{Pd-top}
P_d^{(d-1)}(t)=(-1)^m\,\frac{8a_d}{d}.
\end{equation}

Together with \eqref{S-spectrum}, this yields
\[
\widehat S=P_d\quad \text{on $(0,1)$}.
\]

\subsection{The function $A$ and the polynomials $Q_d$, $R_d$}
Differentiating \eqref{S-log}, define the even meromorphic function
\begin{equation}\label{def-A}
A(z)=\frac{\pi i}{4}\,S'(z) =\frac12\Bigl(\frac{\Phi'(z)}{\Phi(z)}+
\frac{\Phi'(-z)}{\Phi(-z)}\Bigr).
\end{equation}
Its poles are simple and
\begin{equation}\label{A-residue}
A(z)=\sum_{k=1}^\infty(-1)^{k-1}\,\frac{\tau_k}{z^2-\tau_k^2},\quad
\res_{z=\tau_k}A(z)=\frac{(-1)^{k-1}}2.
\end{equation}
Here the series converges locally uniformly away from the poles: its leading
term as $k\to\infty$ is $(-1)^k/\tau_k$, while the remainder is
$O_K(\tau_k^{-3})$ on every compact set $K\subset \mathbb{C}$.

Using \eqref{def-Pd}, define
\begin{equation}\label{Lambda-def}
\Lambda_{2n+1}=\frac18\,(-1)^n(2n+1)!\,p_n,\quad 0\le n\le m,
\end{equation}
The interpretation of the coefficients $\Lambda_{2n+1}$ in terms of
regularized power sums over the zeros will be given in
Subsection~\ref{sec-fp-zeta}.

Define the even real polynomial
\begin{equation}\label{def-Qd}
Q_d(z)=\sum_{n=0}^{m-1}\Lambda_{2n+1}z^{2(m-1-n)},
\end{equation}
with the convention $Q_1=0$, and set
\begin{equation}\label{def-Rd}
R_d(z)=a_d+z^2Q_d(z).
\end{equation}

\subsection{Spectral gap for $\Theta_d$}\label{subsec-Theta-d}
Consider
\[
\Theta_d(z)=z^dA(z)-zQ_d(z)-\frac{a_d}{z}.
\]
Its boundary values are
\[
\Theta_d^\pm(x)= x^dA^\pm(x)-xQ_d(x)-\frac{a_d}{x\pm i0}.
\]

We now establish the spectral gap
\begin{equation}\label{Theta-gap}
\supp\widehat{\Theta_d^+}\subset[1,\infty),\quad
\supp\widehat{\Theta_d^-}\subset(-\infty,-1].
\end{equation}
From \eqref{S-spectrum} and \eqref{def-Pd}, we may write
\[
\widehat S(t)=t_+^0P_d(t)+\gamma\delta_0+T(t),\quad \supp T\subset[1,\infty).
\]
The term $\gamma\delta_0$ disappears after multiplying $\widehat S$ by $t$,
since $t\delta_0=0$. For the polynomial part we have the exact decomposition
\begin{equation}\label{spectral-polynomial-part}
\widehat{x^dA^+}(t)= \frac{\pi i}{4}\,i^{d+1} \sum_{n=0}^m p_n\,\partial_t^d
t_+^{2n+1}+T_1(t), \quad \supp T_1\subset[1,\infty).
\end{equation}
Indeed, $T_1=(\pi i/4)i^{d+1}\,\partial_t^d(tT)$, so differentiation does not
move its support to the left. For $n=m$, Lemma~\ref{lem-truncated-power} and
\eqref{Pd-top} give
\[
\frac{\pi i}{4}\,i^{d+1}p_m d!\,\mathbf{1}_{(0,\infty)}= -2\pi
ia_d\,\mathbf{1}_{(0,\infty)}.
\]
This term cancels the Fourier transform of $-a_d/(x+i0)$ by
\eqref{boundary-Fourier}.

Now let $n<m$ and $r=2(m-n)-1$. The corresponding term in
\eqref{spectral-polynomial-part} supported at the origin is
\[
\frac{\pi i}{4}\,(-1)^{m+1}(2n+1)!\,p_n\delta_0^{(r)}.
\]
The coefficient of $z^r$ in $zQ_d(z)$ is
\[
q_n=\frac18\,(-1)^n(2n+1)!\,p_n.
\]
Therefore,
\[
\widehat{-q_nx^r}= -2\pi q_ni^r\delta_0^{(r)}= -\frac{\pi
i}{4}\,(-1)^{m+1}(2n+1)!\,p_n\delta_0^{(r)},
\]
and these terms cancel. The remainder in \eqref{spectral-polynomial-part} is
supported in $[1,\infty)$. Hence,
$\supp\widehat{\Theta_d^+}\subset[1,\infty)$. The second inclusion in
\eqref{Theta-gap} follows by complex conjugation and reflection $t\mapsto-t$.

\subsection{The polynomial defect $H_d$}
Introduce the defect
\begin{equation}\label{def-Hd}
H_d(z)=\varphi(z)\Theta_d(z)+\frac{a_d}{z}.
\end{equation}
The poles of $\Theta_d$ at $\pm\tau_k$ are canceled by the simple zeros of
$\varphi$. At the origin,
$\varphi(z)=1+O(z^2)$ and $\Theta_d(z)=-a_d/z+O(z)$, so this singularity is
also removable. Hence $H_d$ is entire and odd.

Since $|x|^{d-1}\varphi(x)\in L^1(\mathbb{R})$ and $\varphi$ is bounded on
$\mathbb{R}$, we have $\varphi\in\PW_1^1(\mathbb R)$, and therefore
$\supp\widehat\varphi\subset[-1,1]$. By Lemma~\ref{lem-multiplier},
\eqref{Theta-gap}, and \eqref{boundary-Fourier},
\[
\supp\widehat{H_d^+}\subset[0,\infty),\quad
\supp\widehat{H_d^-}\subset(-\infty,0].
\]
The two boundary values coincide, as distributions, with the restriction to
$\mathbb R$ of the same entire function $H_d$. Hence,
\[
\supp\widehat H_d\subset[0,\infty)\cap(-\infty,0]=\{0\}.
\]
A tempered distribution with such spectrum is a polynomial. Therefore the
entire function $H_d$ is also a polynomial, and by oddness
\[
H_d(z)=z\Pi_d(z)
\]
for some even polynomial $\Pi_d$.

\subsection{Vanishing of $H_d$}
Letting $z\to\tau_k$ in \eqref{def-Hd} and using the residue
\eqref{A-residue}, we obtain
\begin{equation}\label{Pi-at-zero}
\tau_k^{d-1}\varphi'(\tau_k)
=2(-1)^{k-1}\Bigl(\Pi_d(\tau_k)-\frac{a_d}{\tau_k^2}\Bigr).
\end{equation}
The function $g(z)=z^{d-1}\varphi(z)$ belongs to
$\PW_1^1(\mathbb R)$ and
$g'(\tau_k)=\tau_k^{d-1}\varphi'(\tau_k)$. Since $\tau_k\to\infty$, we can
choose a subsequence $\tau_{k_j}$ such that
$\tau_{k_{j+1}}-\tau_{k_j}\ge1$. Then Lemma~\ref{lem-PP} gives
\[
\sum_j|g'(\tau_{k_j})|<\infty.
\]

On the other hand, if $\Pi_d\not\equiv0$, then
$|\Pi_d(x)-a_d/x^2|\ge C$ for all sufficiently large $x$. Hence the right-hand
side of \eqref{Pi-at-zero} cannot be absolutely summable along
$\tau_{k_j}$. Therefore, $\Pi_d\equiv0$.

Thus,
\begin{equation}\label{phi-Theta}
\varphi(z)\Theta_d(z)=-\frac{a_d}{z}.
\end{equation}

\subsection{Final functional equation}
Using \eqref{def-A} and \eqref{def-Rd}, rewrite \eqref{phi-Theta} as
\begin{equation}\label{quadratic-FE}
z^{d+1}\bigl(\Phi'(z)\Phi(-z)+\Phi'(-z)\Phi(z)\bigr)
-2R_d(z)\Phi(z)\Phi(-z)=-2a_d.
\end{equation}
This is the quadratic functional equation \eqref{intro-FE}.

\section{Proof of Theorem \ref{thm-main}~(b)}%\label{sec-ode}

Again, we split the proof into short steps.

\subsection{Relation at the zeros of $\Phi$}
Differentiate \eqref{quadratic-FE} and set $z=s$, where $\Phi(s)=0$. The mixed
terms produced by differentiation cancel, and all terms containing $\Phi(s)$
vanish. Since the zeros of $\Phi(z)$ and $\Phi(-z)$ are disjoint,
$\Phi(-s)\ne0$. Hence,
\begin{equation}\label{zero-relation}
s^{d+1}\Phi''(s)
+\bigl((d+1)s^d-2R_d(s)\bigr)\Phi'(s)=0.
\end{equation}

\subsection{The polynomial $\mathcal P_d$}
Define
\[
\mathcal D_df= z^{d+1}f''+\bigl((d+1)z^d-2R_d(z)\bigr)f',\quad \mathcal
P_d(z)=-\frac{\mathcal D_d\Phi(z)}{\Phi(z)}.
\]
By \eqref{zero-relation}, all possible poles at the simple zeros of $\Phi$ are
removable. Hence $\mathcal P_d$ is entire. Lemma~\ref{lem-logder} shows that
$\mathcal P_d$ is a polynomial. Since $\Phi$ and $R_d$ are real-valued on
$\mathbb R$, this polynomial has real coefficients.

\subsection{Equation for $\widetilde{\Phi}$ and the polynomial $K_d$}
Write
\[
R_d(z)=\sum_{j=0}^m r_jz^{2j},\quad r_0=a_d,
\]
and define the odd rational function
\[
h_d(z)=-\sum_{j=0}^m\frac{r_j}{d-2j}\,z^{-(d-2j)}.
\]
Then $h_d'(z)=R_d(z)/z^{d+1}$. In the punctured plane, define the modified
function $\widetilde{\Phi}$ by
\[
\widetilde{\Phi}(z)=e^{-h_d(z)}\Phi(z).
\]
The functional equation \eqref{quadratic-FE} immediately gives
\[
\widetilde{\Phi}'(z)\widetilde{\Phi}(-z)+
\widetilde{\Phi}'(-z)\widetilde{\Phi}(z)= -\frac{2a_d}{z^{d+1}}.
\]
If we set $Y_1(z)=\widetilde{\Phi}(z)$ and $Y_2(z)=\widetilde{\Phi}(-z)$, then,
with the convention $W(Y_1,Y_2)=Y_1Y_2'-Y_1'Y_2$ for the Wronskian,
\[
W(Y_1,Y_2)=\frac{2a_d}{z^{d+1}}\ne0.
\]
Thus $Y_1$, $Y_2$ form a fundamental system for a unique monic second-order
equation. Abel's formula gives the coefficient $(d+1)/z$ at the first
derivative. Reflection $z\mapsto-z$ interchanges the fundamental solutions,
so the second coefficient is even. Denote it by $q_d(z)$. Then $q_d$ is
holomorphic in $\mathbb C\setminus\{0\}$ and $\widetilde{\Phi}$ satisfies
\begin{equation}\label{gauge-ODE}
\widetilde{\Phi}''(z)+\frac{d+1}{z}\,\widetilde{\Phi}'(z)+
q_d(z)\widetilde{\Phi}(z)=0.
\end{equation}

Introduce the polynomial
\begin{equation}\label{def-Kd}
K_d(z)=\mathcal P_d(z)+R_d'(z).
\end{equation}
Substituting $\Phi=e^{h_d}\widetilde{\Phi}$ into
$\mathcal D_d\Phi+\mathcal P_d\Phi=0$ gives
\begin{equation}\label{q-before-degree}
q_d(z)=\frac{K_d(z)}{z^{d+1}}
-\frac{R_d(z)^2}{z^{2d+2}}.
\end{equation}
Since $q_d$, $R_d$, and $z^{d+1}$ are even, the polynomial $K_d$ is even.

Since $\deg R_d'\le d-2$, the polynomials $K_d$ and $\mathcal P_d$ have the
same degree and the same leading coefficient. In the next steps we determine
the degree and leading coefficient of $K_d$, and then find its remaining
coefficients.

\subsection{Liouville transform and zero density}
Apply the Liouville transform directly to \eqref{gauge-ODE} on the positive
half-line. Since the coefficient at $\widetilde{\Phi}'$ is $(d+1)/x$, up to a
constant factor,
\begin{equation}\label{def-u}
u(x)=x^{(d+1)/2}\widetilde{\Phi}(x)
=x^{(d+1)/2}e^{-h_d(x)}\Phi(x).
\end{equation}
By Lemma~\ref{lem-Liouville},
\[
u''(x)+V(x)u(x)=0,\quad x>0,
\]
where
\begin{equation}\label{V-formula}
V(x)=q_d(x)-\frac{d^2-1}{4x^2}
=\frac{K_d(x)}{x^{d+1}}
-\frac{R_d(x)^2}{x^{2d+2}}
-\frac{d^2-1}{4x^2}.
\end{equation}
Since $\deg R_d\le d-1$, \eqref{V-formula} gives
\begin{equation}\label{V-asymp-K}
V(x)=\frac{K_d(x)}{x^{d+1}}+O(x^{-2}),\quad x\to+\infty.
\end{equation}
The factor in \eqref{def-u} is positive for $x>0$, so the positive zeros of
$u$ are exactly the positive zeros of $\Phi$, namely
$\tau_1,\tau_3,\tau_5,\ldots.$ Hence, by \eqref{tau-props},
\begin{equation}\label{u-zero-density}
n_u(X)=\frac{X}{2\pi}+o(X),\quad X\to\infty.
\end{equation}

\subsection{Lower bound for the degree of $K_d$}
We first prove
\begin{equation}\label{degree-lower}
\deg K_d\ge d+1.
\end{equation}
If $K_d\equiv0$ or $\deg K_d\le d$, then \eqref{V-asymp-K} gives
$V(x)\to0$. For every $\varepsilon>0$, we have $V(x)\le\varepsilon$ for all
sufficiently large $x$. Lemma~\ref{lem-zero-count}~(a) gives
\[
\limsup_{X\to\infty}\frac{n_u(X)}X\le\frac{\sqrt\varepsilon}{\pi}.
\]
Letting $\varepsilon\to0$ contradicts \eqref{u-zero-density}.

\subsection{Degree of the polynomial $K_d$}
Let $N=\deg K_d$, let $c_N\ne0$ be the leading coefficient of $K_d$, and set
$\kappa=N-d-1$. By \eqref{degree-lower}, $\kappa\ge0$ and
\[
V(x)=c_Nx^\kappa(1+o(1)),\quad x\to+\infty.
\]
If $c_N<0$, then $V(x)<0$ for all sufficiently large $x$, and
Lemma~\ref{lem-zero-count}~(a) shows that there can be at most one zero there,
contradicting \eqref{u-zero-density}. Thus $c_N>0$.

If $\kappa>0$, then on $[X,2X]$, for large $X$,
\[
V(x)\ge\frac{c_N}{2}\,X^\kappa.
\]
By Lemma~\ref{lem-zero-count}~(b),
\[
n_u(2X)-n_u(X)\ge \frac1\pi\,\sqrt{\frac{c_N}{2}}\,X^{1+\kappa/2}-2,
\]
which grows superlinearly and contradicts \eqref{u-zero-density}. Therefore
$\kappa=0$ and
\begin{equation}\label{degree-exact}
\deg K_d=d+1.
\end{equation}

\subsection{Leading coefficient of $K_d$}
Let $c>0$ be the leading coefficient of $K_d$. By \eqref{degree-exact} and
\eqref{V-asymp-K},
\[
V(x)=c+o(1),\quad x\to+\infty.
\]
For every $0<\varepsilon<c$ and all sufficiently large $x$,
\[
c-\varepsilon\le V(x)\le c+\varepsilon.
\]
Lemma~\ref{lem-zero-count} gives
\[
\frac{\sqrt{c-\varepsilon}}\pi
\le\liminf_{X\to\infty}\frac{n_u(X)}X
\le\limsup_{X\to\infty}\frac{n_u(X)}X
\le\frac{\sqrt{c+\varepsilon}}\pi.
\]
Using \eqref{u-zero-density} and then letting $\varepsilon\to0$, we get
\[
\frac{\sqrt c}{\pi}=\frac1{2\pi}.
\]
Hence,
\[
c=\frac14.
\]

\subsection{Remaining coefficients of $K_d$}
Thus, $K_d$ is an even polynomial of degree $d+1$ with leading coefficient
$1/4$. It remains to determine the remaining coefficients of $K_d$.

In the disk $|z|<\tau_1$, expand
\[
A(z)=\sum_{n=0}^\infty\ell_{2n+1}z^{2n},\quad
\ell_{2n+1}=\sum_{k=1}^\infty\frac{(-1)^k}{\tau_k^{2n+1}}.
\]
For $n=0$ the series converges by the Leibniz criterion, and for $n\ge1$ it
converges absolutely.

Let $L=\Phi'/\Phi$. From \eqref{def-A},
\[
\frac{L(z)+L(-z)}2=A(z).
\]
On the other hand,
\[
K_d=-z^{d+1}\,\frac{\Phi''}{\Phi}- (d+1)z^dL+2R_dL+R_d'.
\]
The first two terms do not contribute to even powers of degree at most $d-1$,
and $R_d'$ is odd. Therefore these coefficients coincide with the
corresponding coefficients of $2R_dA$. Hence,
\begin{equation}\label{Kd-formula}
K_d(z)=\frac14\,z^{d+1}+
2\sum_{n=0}^m\biggl(\sum_{j=0}^nr_j\ell_{2(n-j)+1}\biggr)z^{2n}.
\end{equation}
Also, directly from the definition of $Q_d$,
\[
r_j=\frac18\,(-1)^{m-j}\bigl(2(m-j)+1\bigr)!\,p_{m-j}, \quad 1\le j\le m.
\]
For $j=0$, the same formula follows from \eqref{Pd-top}. Thus, to construct
the equation, it is enough to know the finite sets
$p_0,\ldots,p_m$ and $\ell_1,\ell_3,\ldots,\ell_d$.

\subsection{Final differential equation}
Substituting $\mathcal P_d=K_d-R_d'$ into
$\mathcal D_d\Phi+\mathcal P_d\Phi=0$, we obtain
\begin{equation}\label{explicit-ODE}
z^{d+1}\Phi''(z)
+\bigl((d+1)z^d-2R_d(z)\bigr)\Phi'(z)
+\bigl(K_d(z)-R_d'(z)\bigr)\Phi(z)=0.
\end{equation}
This is the required differential equation \eqref{intro-ODE}.

\section{Proof of Theorem \ref{thm-main}~(c)}\label{sec-Phi}

By part~\textup{(b)} already proved,
\[
K_d(x)=\frac14\,x^{d+1}+O(x^{d-1}),
\]
and $\deg R_d\le d-1$. Hence \eqref{V-formula} gives
\[
V(x)=\frac14+O(x^{-2}),\quad x\to+\infty.
\]
For a solution of $u''+Vu=0$, introduce the energy function
\[
\mathcal E(x)=u'(x)^2+\frac14\,u(x)^2.
\]
Then
\[
\mathcal E'(x)=-2\Bigl(V(x)-\frac14\Bigr)u(x)u'(x),
\]
and for all sufficiently large $x$,
\begin{equation}\label{energy}
|\mathcal E'(x)|\le Cx^{-2}\mathcal E(x).
\end{equation}
Hence $u(x),u'(x)=O(1)$. Applying the same argument to the second solution
corresponding to $\Phi(-x)$ gives
\[
\Phi(x)=O\bigl(|x|^{-(d+1)/2}\bigr),\quad |x|\to\infty.
\]
Since, by Lemma \ref{lem-Phi}, $\Phi$ has finite exponential type, the
Paley--Wiener--Schwartz theorem implies that $\widehat{\Phi}$ has compact
support.

Apply the Fourier transform to the differential equation
\eqref{explicit-ODE}. Up to a nonzero constant factor, the coefficient of the
highest derivative $\partial_t^{d+1}\widehat\Phi$ is $1/4-t^2$. Therefore, on
each of the intervals $(-\infty,-1/2)$ and $(1/2,\infty)$, the distribution
$\widehat\Phi$ is a solution of a regular linear differential equation and
hence is represented there by a smooth function. Since $\widehat\Phi$ has
compact support, uniqueness of solutions gives
\begin{equation}\label{Phi-spectrum}
\supp\widehat\Phi\subset[-1/2,1/2].
\end{equation}
Again by the Paley--Wiener--Schwartz theorem, \eqref{Phi-spectrum} implies
that the exponential type of $\Phi$ is at most $1/2$. On the other hand,
$\Phi(z)$ and $\Phi(-z)$ have the same type, while
$\varphi(z)=\Phi(z)\Phi(-z)$ has exact exponential type $1$. Therefore the
type of $\Phi$ is at least $1/2$, and hence equals $1/2$.

\section{The case $d=1$}\label{sec-d1}

If $d=1$, then
\[
m=0,\quad V_1=2,\quad a_1=\frac1{4\mathcal C_1},\quad Q_1=0,\quad R_1=a_1.
\]
The functional equation \eqref{quadratic-FE} becomes
\begin{equation}\label{func-eq}
z^2\bigl(\Phi'(z)\Phi(-z)+\Phi'(-z)\Phi(z)\bigr)
-\frac1{2\mathcal C_1}\,\Phi(z)\Phi(-z)
=-\frac1{2\mathcal C_1}.
\end{equation}

Set $\ell_1=\Phi'(0)$. Formula \eqref{Kd-formula} gives
\[
K_1(z)=\frac14\,z^2+2a_1\ell_1.
\]
Hence,
\begin{equation}\label{diff-eq}
z^2\Phi''(z)+\Bigl(2z-\frac1{2\mathcal C_1}\Bigr)\Phi'(z)
+\Bigl(\frac{z^2}{4}+\frac{\ell_1}{2\mathcal C_1}\Bigr)\Phi(z)=0.
\end{equation}

After changing the normalization from type~$1$ to type~$\pi$, equations
\eqref{func-eq} and \eqref{diff-eq} coincide with equations (3.3) and (1.5),
respectively, in \cite{Bo25}.

We note that \cite{Bo25} also contains a substantially different functional
equation. For $d=1$,
\[
h_1(z)=-\frac1{4\mathcal C_1z},\quad \widetilde{\Phi}(z)=e^{1/(4\mathcal
C_1z)}\Phi(z),
\]
and this equation takes the form
\begin{equation}\label{d1-inversion}
\widetilde{\Phi}(z)=
\frac{
e^{-i\pi/4}\widetilde{\Phi}\bigl((2i\mathcal C_1z)^{-1}\bigr)
+
e^{i\pi/4}\widetilde{\Phi}\bigl(-(2i\mathcal C_1z)^{-1}\bigr)}
{2\sqrt{\mathcal C_1}\,z},
\quad z\ne0.
\end{equation}
It corresponds to equation (1.6) in \cite{Bo25} and comes from the inversion
symmetry of the differential equation for $\widetilde{\Phi}$.

For $d\ge3$, such a symmetry is no longer apparent. Indeed, the equation for
$\widetilde{\Phi}$ is
\[
\widetilde{\Phi}''+\frac{d+1}{z}\,\widetilde{\Phi}'+q_d(z)\widetilde{\Phi}=0,\quad
q_d(z)=\frac{K_d(z)}{z^{d+1}}- \frac{R_d(z)^2}{z^{2d+2}}.
\]
The transformation $\widetilde{\Phi}(z)\mapsto z^{-d}\widetilde{\Phi}(c/z)$
would preserve this equation only if
\[
q_d(z)=\frac{c^2}{z^4}\,q_d(c/z).
\]
For $d=1$, this condition is satisfied for a suitable choice of $c$, and this
is precisely what underlies \eqref{d1-inversion}. However, for $d\ge3$, the function $q_d$ contains the
nonzero term $-a_d^2z^{-2d-2}$, which after inversion produces the positive
power $z^{2d-2}$.

\section{The case $d=3$: differential equation and numerical
algorithm}\label{sec-d3}

\subsection{Differential equation}
For $d=3$,
\[
P_3(t)=p_0+p_1t^2,\quad p_1=-\frac{4a_3}{3},\quad
Q_3(z)=\frac{p_0}{8}=\Lambda_1.
\]
Therefore,
\begin{equation}\label{d3-FE}
z^4\bigl(\Phi'(z)\Phi(-z)+\Phi'(-z)\Phi(z)\bigr)
-2(a_3+\Lambda_1z^2)\Phi(z)\Phi(-z)=-2a_3.
\end{equation}

Let
\[
\Phi(z)=1+c_1z+c_2z^2+c_3z^3+\ldots.
\]
Comparing the coefficients of $z^2$ in \eqref{d3-FE} gives
\begin{equation}\label{d3-c2}
2c_2-c_1^2=-\frac{\Lambda_1}{a_3}.
\end{equation}
Also,
\[
\ell_1=c_1,\quad \ell_3=3c_3-3c_1c_2+c_1^3.
\]
By \eqref{Kd-formula},
\[
K_3(z)=\frac14\,z^4+\beta_2z^2+2a_3c_1,\quad
\beta_2=2(a_3\ell_3+\Lambda_1\ell_1).
\]
Using \eqref{d3-c2},
\[
\beta_2=6a_3c_3+4\Lambda_1c_1-2a_3c_1c_2.
\]
Hence,
\begin{equation}\label{d3-ODE}
z^4\Phi''(z)+ (4z^3-2\Lambda_1z^2-2a_3)\Phi'(z)
+\Bigl(\frac14\,z^4+\beta_2z^2-2\Lambda_1z+2a_3c_1\Bigr)\Phi(z)=0.
\end{equation}
For the differential equation \eqref{gauge-ODE}, the coefficient $q_3$ has
the form
\[
q_3(z)=\frac14+\frac{\beta_2}{z^2}
+\frac{2a_3c_1-\Lambda_1^2}{z^4}
-\frac{2a_3\Lambda_1}{z^6}-\frac{a_3^2}{z^8}.
\]

Thus the problem is reduced to determining four parameters,
\[
a_3,\quad \Lambda_1,\quad \beta_2,\quad c_1.
\]
Once $a_3$ is known, \eqref{def-ad} gives
\[
\mathcal C_3=\frac{3}{8\pi a_3}.
\]

In the one-dimensional case, after the Fourier transform the differential
equation in~\cite{Bo25} leads to a tridiagonal spectral problem in the basis of
Legendre polynomials and hence to a high-precision numerical algorithm. For
$d=3$ an analogous scheme appears: the transformed equation has order four,
and the corresponding matrix becomes seven-diagonal.

\subsection{Spectral boundary-value problem}
For $d=3$, Theorem~\ref{thm-main}~\textup{(c)} gives
\[
\Phi(x)=O(|x|^{-2}),\quad |x|\to \infty,
\]
so $\Phi\in\PW_{1/2}^{1}(\mathbb R)$. In particular,
$\widehat\Phi\in C_0(\mathbb R)$ and
$\supp\widehat\Phi\subset[-1/2,1/2]$.

Set
\[
h(x)=\widehat\Phi(x/2),\quad x\in[-1,1].
\]
Then $h$ is continuous,
\begin{equation}\label{d3-dirichlet}
h(-1)=h(1)=0,
\end{equation}
and, since $\Phi$ is real-valued, $h$ is Hermitian:
\begin{equation}\label{d3-reality}
h(-x)=\overline{h(x)}.
\end{equation}

Inside $(-1,1)$, the Fourier transform of equation~\eqref{d3-ODE}, written in
terms of $h$, gives
\begin{equation}\label{d3-fourier-ODE}
(1-x^2)h^{(4)}-4xh^{(3)}-\beta_2h''
+i\Lambda_1(xh''+h')
-\frac{ia_3}{4}\,xh+\frac{a_3c_1}{2}\,h=0.
\end{equation}
Since this equation is regular on $(-1,1)$, we have $h\in C^4(-1,1)$. By the
Fourier inversion formula,
\[
\int_{-1}^1 h(x)e^{itx}\,dx
=4\pi\Phi(2t)=O(t^{-2}),\quad t\to\infty.
\]
Hence, by Lemma~\ref{lem-boundary-regularity}, $h$ extends to an entire
function. In particular, all boundary derivatives $h^{(k)}(\pm1)$ exist.

Now return to the Fourier transform of equation~\eqref{d3-ODE} on the whole
real line in the sense of distributions, extending $h$ by zero outside
$[-1,1]$. Differentiation produces boundary $\delta$-terms at $s=\pm1$.
Using
\begin{align*}
(1-x^2)\delta_s'&= 2s\delta_s,\\
(1-x^2)\delta_s''&= 4s\delta_s'-2\delta_s,\\
(1-x^2)\delta_s'''&= 6s\delta_s''-6\delta_s',
\end{align*}
we find that the boundary part of the left-hand side of
\eqref{d3-fourier-ODE} at $x=s$ is, up to a common nonzero factor,
\[
2s h(s)\delta_s''
+(2-\beta_2+i\Lambda_1s)h(s)\delta_s'
+\bigl(-2s h''(s)+(2-\beta_2+i\Lambda_1s)h'(s)\bigr)\delta_s.
\]
Since $h(s)=0$, the first two terms vanish, and setting the coefficient of
$\delta_s$ equal to zero gives
\[
-2s h''(s)+(2-\beta_2+i\Lambda_1s)h'(s)=0,\quad s=\pm1.
\]

For the solution of the spectral boundary-value problem corresponding to the
extremal function, there is an additional phase condition
\begin{equation}\label{d3-phase}
h'(1)=-ih'(-1).
\end{equation}
Indeed, set $A=h'(1)$. Then \eqref{d3-reality} gives
$h'(-1)=-\overline A$. If $A=0$, then \eqref{d3-phase} is immediate. Thus,
assume $A\ne0$.

From \eqref{d3-dirichlet}, the Fourier inversion formula, and integration by
parts, we obtain
\begin{equation}\label{d3-Phi-asymptotic}
\Phi(z)=
\frac{h'(1)e^{iz/2}-h'(-1)e^{-iz/2}}{\pi z^2}
+O\bigl(|z|^{-3}e^{|\!\Im z|/2}\bigr),
\quad |z|\to\infty.
\end{equation}
Hence, by \eqref{varphi-Phi},
\begin{equation}\label{varphi-cos}
x^4\varphi(x) =\frac{2|A|^2}{\pi^2}\bigl(\cos x+q\bigr)+O(x^{-1}), \quad
q=\frac{\Re A^2}{|A|^2},
\end{equation}
where $|q|\le 1$.

We show that $q=0$. Choose an even nonnegative function
$\eta\in\mathcal S(\mathbb R)$ such that $\supp\widehat\eta\subset(-1,1)$ and
$\widehat\eta(0)=1$, and set $\eta_R(x)=R^{-1}\eta(x/R)$ for $R>1$. Then
\[
\widehat{\eta_R}(\xi)=\widehat\eta(R\xi), \quad
\supp\widehat{\eta_R}\subset(-1/R,1/R).
\]
Using \eqref{sigma-P}, we obtain
\begin{equation}\label{sigma-0}
\int_{\mathbb R}\sigma(x)\eta_R(x)\,dx =\frac1{4\pi}\int_{\mathbb R}
P_3(u)\widehat{\eta_R}(u)\,du =\frac1{4\pi R}\int_{\mathbb R}
P_3(u/R)\widehat\eta(u)\,du =O(R^{-1}).
\end{equation}

Set $g(x)=\cos x+q$ and, for small $\varepsilon>0$, define the
$2\pi$-periodic set
\[
E_\varepsilon= \{x\in\mathbb R\colon |g(x)|\le\varepsilon\}.
\]
By \eqref{varphi-cos}, there exists $X_\varepsilon>0$ such that
\[
\sigma(x)=\sign g(x),\quad |x|>X_\varepsilon,\quad x\notin E_\varepsilon.
\]
Therefore,
\[
\biggl|\int_{\mathbb R}\bigl(\sigma(x)-\sign g(x)\bigr)\eta_R(x)\,dx\biggr|\le
2\int_{|x|\le X_\varepsilon}\eta_R(x)\,dx+ 2\int_{E_\varepsilon}\eta_R(x)\,dx.
\]
The first term tends to zero as $R\to\infty$, while the second equals
$2\int_{\mathbb R}\mathbf 1_{E_\varepsilon}\eta_R$. Thus, using
\eqref{sigma-0} and applying Lemma~\ref{lem-periodic-averaging} to
$\sign g$ and $\mathbf 1_{E_\varepsilon}$, we get
\[
\biggl|\frac1{2\pi}\int_0^{2\pi}\sign(\cos x+q)\,dx\biggr|\le
2\,\frac1{2\pi}\int_0^{2\pi}\mathbf 1_{E_\varepsilon}(x)\,dx
\]
or
\[
|2\arcsin q|\le |E_\varepsilon\cap[0,2\pi]|.
\]
Letting $\varepsilon\to0$, we get $q=0$, because the zero set of $g$ on
$[0,2\pi]$ has measure zero. Hence $\Re A^2=0$ and
$A=\pm i\overline A$.

It remains to select one of the two phase branches. We use
Lemma~\ref{lem-HB}. From the integral representation
\begin{equation}\label{d3-Phi-inversion}
\Phi(z)=\frac1{4\pi}\int_{-1}^1 h(t)e^{izt/2}\,dt
\end{equation}
and \eqref{d3-dirichlet}, two integrations by parts as $y\to+\infty$ give
\begin{align*}
E(iy)&= -\frac{\overline A+iA}{\pi y^2}\,e^{y/2} +O(e^{y/2}y^{-3}),\\
E^\#(iy)&= -\frac{\overline A-iA}{\pi y^2}\,e^{y/2} +O(e^{y/2}y^{-3}).
\end{align*}
The branch $A=-i\overline A$ contradicts the inequality
$|E(z)|<|E^\#(z)|$, $\Im z>0$. Hence $A=i\overline A$, which gives
\eqref{d3-phase}.

\subsection{Seven-diagonal scheme}
Let $P_n^{(\alpha,\beta)}$ be the Jacobi polynomials \cite{Se74}.

\subsubsection{Basis and matrix operators}
Set
\[
e_n(x)=\frac{n+1}{2}\,(1-x^2)P_{n-1}^{(1,1)}(x), \quad n\in\mathbb N.
\]
These functions automatically satisfy \eqref{d3-dirichlet}. Write
\[
\nu_n=n(n+1),\quad \mu_n=(n-1)(n+2).
\]
The standard formulas for Jacobi polynomials give
\begin{align*}
\partial_x\bigl((1-x^2)P_{n-1}^{(1,1)}(x)\bigr)
&=-2nP_n^{(0,0)}(x),\\
\bigl((1-x^2)\,\partial_x^2-4x\,\partial_x\bigr)P_{n-1}^{(1,1)}
&=-\mu_nP_{n-1}^{(1,1)}.
\end{align*}
Hence,
\[
e_n'=-\nu_nP_n^{(0,0)},\quad
\bigl((1-x^2)\,\partial_x^4-4x\,\partial_x^3\bigr)e_n =-\mu_ne_n''.
\]
Also, with the convention $e_0=0$,
\begin{equation}\label{d3-x-recurrence}
xe_n=\frac{n}{2n+1}\,e_{n+1} +\frac{n+1}{2n+1}\,e_{n-1}.
\end{equation}

Consider the Hilbert space
\[
\mathcal X=L^2\Bigl((-1,1),\frac{dx}{1-x^2}\Bigr).
\]
The system $(e_n)_{n\ge1}$ is a complete orthogonal basis in this space.
Indeed, the orthogonality of the Jacobi polynomials $P_{n-1}^{(1,1)}$ with
weight $1-x^2$ gives
\[
\langle e_n,e_m\rangle_{\mathcal X}
=
\frac{2n(n+1)}{2n+1}\,\delta_{nm},
\]
and completeness follows from completeness of the Jacobi polynomials.

Since $h$ is entire and $h(\pm1)=0$, the function
\[
G(z)=\frac{h(z)}{1-z^2}
\]
is also entire. Therefore $h=(1-x^2)G\in\mathcal X$, and using the Hermitian
symmetry \eqref{d3-reality},
\begin{equation}\label{d3-h-expansion}
h(x)=\sum_{n\ge1}i^{\,n-1}y_ne_n(x), \quad y_n\in\mathbb R.
\end{equation}
Equivalently,
\[
G(x)=\sum_{n\ge1}\frac{n+1}{2}\,
i^{\,n-1}y_nP_{n-1}^{(1,1)}(x).
\]
Since $G$ is entire, \cite[Th.~9.1.1]{Se74} gives
\[
\lim_{n\to\infty}|y_n|^{1/n}=0.
\]
Consequently, for every $\rho>1$ there is $C_\rho>0$ such that
\begin{equation}\label{d3-coefficient-decay}
|y_n|\le C_\rho\rho^{-n},\quad n\ge1.
\end{equation}
Thus the series used below may be differentiated term by term as many times as
needed, and we may substitute $x=\pm1$ into them.

\begin{rem}
This choice of basis is naturally consistent with the one-dimensional case.
For $d=1$ one obtains the Jacobi polynomials $P_n^{(0,0)}$, while for $d=3$
one obtains the functions $(1-x^2)P_n^{(1,1)}$. This suggests the natural
basis
\[
(1-x^2)^mP_n^{(m,m)}(x),\quad d=2m+1,
\]
for general odd dimensions. We do not construct the corresponding
finite-dimensional spectral scheme for arbitrary $m$ here.
\end{rem}

Let $J=(-\partial_x^2)^{-1}$ be the inverse of $-\partial_x^2$ with Dirichlet
boundary conditions. Let $W$ be the matrix of $J$ in the basis $(e_n)$. Its
nonzero entries are
\[
\begin{cases}
W_{n,n}=\dfrac2{(2n-1)(2n+3)}, & n\ge1,\\[3mm]
W_{n+2,n}=-\dfrac{n(n+1)}
{(2n+1)(2n+3)(n+2)(n+3)}, & n\ge1,\\[3mm]
W_{n-2,n}=-\dfrac{n(n+1)}
{(2n-1)(2n+1)(n-2)(n-1)}, & n\ge3.
\end{cases}
\]
The matrix $U$ of the operator $J\,\partial_x(x\,\partial_x)$ is given by
\[
\begin{cases}
U_{n+1,n}=-\dfrac{n(n+1)}{(2n+1)(n+2)}, & n\ge1,\\[3mm]
U_{n-1,n}=-\dfrac{n(n+1)}{(2n+1)(n-1)}, & n\ge2.
\end{cases}
\]
Let $X$ be the matrix of multiplication by $x$, given by
\eqref{d3-x-recurrence}, and let $V=WX$.

\subsubsection{Spectral system}
Set
\[
S=\operatorname{diag}(1,i,i^2,\ldots)
\]
and introduce the real matrices
\[
\widetilde W=S^{-1}WS,\quad \widetilde U=iS^{-1}US,\quad \widetilde
V=-iS^{-1}VS.
\]
Since the coefficient vector of $h$ in the basis $(e_n)$ is $Sy$, applying
$J$ to \eqref{d3-fourier-ODE} and introducing the spectral parameter
\[
\lambda=-\frac{a_3c_1}{2}
\]
gives
\begin{equation}\label{d3-matrix-pencil}
\Bigl(D_\mu+\beta_2I+\Lambda_1\widetilde U +\frac{a_3}{4}\,\widetilde V\Bigr)y=
\lambda\widetilde Wy,\quad D_\mu=\operatorname{diag}(\mu_n).
\end{equation}
Thus we obtain a real seven-diagonal generalized spectral problem.

The boundary condition at $x=1$ becomes
\begin{equation}\label{d3-boundary-functional}
\sum_{n\ge1}\nu_n
\bigl(\nu_n-(2-\beta_2+i\Lambda_1)\bigr)
i^{\,n-1}y_n=0.
\end{equation}
The phase condition \eqref{d3-phase} is equivalent to
\begin{equation}\label{d3-phase-functional}
\sum_{n\ge1}(-1)^{\lfloor n/2\rfloor}\nu_ny_n=0.
\end{equation}
We fix the normalization of the eigenvector by
\[
y_1=3\pi,
\]
which is equivalent to $\Phi(0)=1$. Indeed,
\[
\int_{-1}^1e_1(x)\,dx=\frac43,\quad
\int_{-1}^1e_n(x)\,dx=0,\quad n\ge2,
\]
and therefore
\[
\Phi(0) =\frac1{4\pi}\int_{-1}^1h(x)\,dx =\frac{y_1}{3\pi}.
\]

Thus, \eqref{d3-matrix-pencil}, the real and imaginary parts of
\eqref{d3-boundary-functional}, condition \eqref{d3-phase-functional}, and
the normalization $y_1=3\pi$ form an infinite nonlinear system. For the
analysis below, we normalize it by dividing the $n$th equation in
\eqref{d3-matrix-pencil} by $1+\mu_n$. This does not change the set of
solutions. We write the resulting system as
\begin{equation}\label{d3-infinite-system}
\mathcal F(Z)=0,\quad Z=(y,a_3,\Lambda_1,\beta_2,\lambda).
\end{equation}

\subsubsection{Identification of the spectral branch}\label{subsubsec-d3-identification}
We show that the reality and simplicity condition for the zeros of the
reconstructed function allows us to identify a solution of the spectral
system with the extremal branch.

Let $Z=(y,a,\Lambda,\beta,\lambda)$, $a>0$, be a solution of the spectral
system \eqref{d3-infinite-system}. Let $h$ and $\Phi$ be reconstructed from
$y$ by \eqref{d3-h-expansion}, \eqref{d3-Phi-inversion}, and set
$\psi(z)=\Phi(z)\Phi(-z)$. Assume that all zeros of $\psi$ are real and simple.
We show that
\begin{equation}\label{psi-a}
\psi=\varphi_3,\quad a=a_3.
\end{equation}

Set $c=-2\lambda/a$. It follows from \eqref{d3-infinite-system} that
$\Phi$ satisfies \eqref{d3-ODE}. Applying this equation at $z$ and $-z$, we
obtain
\begin{equation}\label{d3-identification-FE}
z^4\bigl(\Phi'(z)\Phi(-z)+\Phi'(-z)\Phi(z)\bigr)
-2(a+\Lambda z^2)\psi(z)=-2a.
\end{equation}

Define $A$ by the right-hand side of \eqref{def-A} and set
\[
\Theta(z)=z^3A(z)-\Lambda z-\frac az.
\]
Then \eqref{d3-identification-FE} is equivalent to
\begin{equation}\label{d3-identification-Theta}
\psi(z)\Theta(z)=-\frac az.
\end{equation}

Let $b=h'(1)$. By \eqref{d3-reality}, \eqref{d3-phase}, and
\eqref{d3-Phi-asymptotic},
\begin{equation}\label{psi-asymp}
b=i\overline b,\quad
\psi(z)=\frac{2|b|^2}{\pi^2z^4}\,\cos z+
O\bigl(|z|^{-5}e^{|\!\Im z|}\bigr).
\end{equation}
Here $b\ne0$, since otherwise the boundary condition and
\eqref{d3-fourier-ODE} successively imply $h^{(k)}(1)=0$ for all $k$, which
contradicts the normalization. In particular,
$\psi(|\Cdot|)\in \PW_1^1(\mathbb R^3)$.

By assumption, the zeros of $\psi$ are real. By
\eqref{d3-identification-Theta}, \eqref{psi-asymp}, and the standard
logarithmic-derivative estimate for an entire function of exponential type
with real zeros, for $y>0$ we have
\[
|e^{-iz}\Theta(z)|
\le C(1+|z|)^N(1+y^{-M}),\quad z=x+iy,
\]
for some $C,N,M>0$. Hence $e^{-iz}\Theta(z)$ has a boundary value in
$\mathcal S'(\mathbb R)$, and the Paley--Wiener theorem for boundary values
(see \cite[Th.~7.4.3]{Ho90}) gives
\begin{equation}\label{d3-identification-gap}
\supp\widehat{\Theta^+}\subset[1,\infty).
\end{equation}

Let $0<t_1<t_2<\ldots$ be the positive zeros of $\psi$. By
\eqref{d3-identification-FE} and the alternating sign of~$\psi$,
\[
\res_{z=t_n}A(z)=\frac{(-1)^{n-1}}2.
\]
By the Mittag--Leffler theorem
(see, e.g., \cite[Ch.~5, Sec.~2]{Ah79}), we obtain an expansion of~$A$
in terms of its poles up to an entire term. The growth
estimates obtained above, after subtraction of the principal parts, show that
this entire term has polynomial growth and hence is a polynomial. Therefore, by \eqref{boundary-Fourier} and the fact that the Fourier
transform of a polynomial is supported at the origin, for $u>0$ we have
\[
\widehat{A^+}(u)=-2\pi S(u),\quad
S(u)=\sum_{n=1}^\infty(-1)^{n-1}\sin(t_nu),
\]
where the series is understood in the Abel sense. From
\eqref{d3-identification-gap} and \eqref{boundary-Fourier}, for $0<u<1$ we
obtain
\[
0=\widehat{\Theta^+}(u)=2\pi i\bigl(S'''(u)+a\bigr).
\]
Consequently,
\[
S(u)=\gamma u-\frac a6\,u^3,\quad |u|<1.
\]
Since
\[
j_{3/2}(x)=3\,\frac{\sin x-x\cos x}{x^3},
\]
we obtain
\[
\sum_{n=1}^\infty(-1)^{n-1}t_n^3j_{3/2}(t_nu)
=\frac3{u^3}\,\bigl(S(u)-uS'(u)\bigr)=a,\quad |u|<1.
\]
The spectral characterization of the zeros in
Subsection~\ref{subsec-zero-characterization} now gives \eqref{psi-a}.

\subsection{Finite truncations}
For the numerical solution, truncate the expansion of $h$ to the first $N$
functions $e_n$ and set
$y^{(N)}=(y_1,\ldots,y_N)$.
Let $\widetilde W^{(N)}$, $\widetilde U^{(N)}$, $\widetilde V^{(N)}$, and
$D_\mu^{(N)}$ denote the principal $N\times N$ sections of the corresponding
infinite matrices.

For given $a_3$, $\Lambda_1$, $\beta_2$, consider the generalized spectral
problem
\begin{equation}\label{d3-finite-pencil}
\Bigl(D_\mu^{(N)}+\beta_2I+\Lambda_1\widetilde U^{(N)}
+\frac{a_3}{4}\,\widetilde V^{(N)}\Bigr)y^{(N)}= \lambda\widetilde
W^{(N)}y^{(N)}.
\end{equation}

When forming the finite system, as in \eqref{d3-infinite-system}, we divide
the $n$th equation in \eqref{d3-finite-pencil} by $1+\mu_n$,
$n=1,\ldots,N$. Together with the finite analogues of the real and imaginary
parts of \eqref{d3-boundary-functional}, condition
\eqref{d3-phase-functional}, and the normalization $y_1=3\pi$, this gives a
system of $N+4$ nonlinear equations
\[
\mathcal F^{(N)}(Z^{(N)})=0,\quad
Z^{(N)}=(y^{(N)},a_3,\Lambda_1,\beta_2,\lambda).
\]
After the solution is found, the parameter $c_1$ is recovered from
\[
c_1=-\frac{2\lambda}{a_3}.
\]

\subsection{Numerical result}
The computations were carried out with $120$ decimal digits of precision. The
finite systems were solved by Newton's method. When the number $N$ of basis
functions was increased, the solution of the previous truncation was used as
the initial approximation. As $N$ increased, all digits displayed below
stabilized rapidly:
\[
\begin{array}{c|r@{}l}
\text{quantity}&
\multicolumn{2}{c}{\text{numerical value}}\\ \hline
a_3&
48&.73902791584226299739\ldots\\
\Lambda_1&
2&.67087119911095542542\ldots\\
\beta_2&
-0&.67468838994063442508\ldots\\
\Phi'(0)=c_1&
-0&.08623217034028118105\ldots\\
\mathcal C_3&
0&.00244908879850930291\ldots
\end{array}
\]
and, in particular, the normalized Nikolskii constant \eqref{def-Lstar} is
\[
\mathcal L^*(3)= 0.14502922550655639263\ldots.
\]
For comparison, \eqref{Da21-bounds} gives
\[
0.125\le\mathcal L^*(3)\le {}_1F_2\Bigl(\frac32;\frac52,\frac52;
-\frac{\beta_3^2}{4}\Bigr) =0.261\ldots.
\]

\subsection{Validation of the numerical scheme}
For a rigorous justification of the numerical result, we first verify the
existence and nondegeneracy of the computed solution of the spectral system,
and then the reality and simplicity of the zeros of the corresponding
function.

\subsubsection{Certification of the spectral solution}
Set
\[
\ell^\infty_2=
\Bigl\{y=(y_n)\colon \|y\|_2=\sup_{n\ge1}2^n|y_n|<\infty\Bigr\}.
\]
By \eqref{d3-coefficient-decay}, the exact solution belongs to
$\ell^\infty_2$.

After dividing the $n$th equation by $1+\mu_n$, all matrix operators define
bounded operators on $\ell^\infty_2$, while the boundary functionals are
continuous because their coefficients have only polynomial growth. Therefore,
system \eqref{d3-infinite-system} defines an analytic mapping
\[
\mathcal F\colon
\ell^\infty_2\times\mathbb R^4\to
\ell^\infty_2\times\mathbb R^4.
\]

To check the nondegeneracy of the spectral branch found numerically, we used
interval arithmetic with the norm
\[
\|(y,a_3,\Lambda_1,\beta_2,\lambda)\|_*=
\max\Bigl\{\sup_{n\ge1}2^n|y_n|,
\frac{|a_3|}{50},\frac{|\Lambda_1|}{3},
|\beta_2|,\frac{|\lambda|}{3}\Bigr\}.
\]
Let $\widetilde Z$ be the numerical solution of the finite system for $N=40$,
extended by zeros in the coordinates $y_n$, $n>40$. As an approximate inverse
to $\nabla\mathcal F(\widetilde Z)$, we used the operator
\[
\widetilde A=
\begin{pmatrix}
D^{-1}&0\\
0&\displaystyle
\operatorname{diag}_{n>40}\bigl((1+\mu_n)/\mu_n\bigr)
\end{pmatrix},
\]
where $D$ is the $44\times44$ Jacobian matrix of the finite system for $N=40$.
For
\[
T(Z)=Z-\widetilde A\mathcal F(Z)
\]
interval arithmetic on the ball $B_{10^{-48}}(\widetilde Z)$ gave
\[
\|T(\widetilde Z)-\widetilde Z\|_*<4.907\cdot10^{-53},\quad
\sup_{B_{10^{-48}}(\widetilde Z)}\|\nabla T\|_*<
3.764\cdot10^{-3}.
\]
Therefore, this ball contains a unique solution $Z_*$ of the system
$\mathcal F(Z)=0$, and the operator $\nabla\mathcal F(Z_*)$ is invertible.

\subsubsection{Verification of the zeros}
It remains to verify the reality and simplicity condition from
Subsection~\ref{subsubsec-d3-identification}. Reconstruct from $Z_*$ the
functions $h$, $\Phi$, and $\psi(z)=\Phi(z)\Phi(-z)$. Interval arithmetic
together with \eqref{d3-h-expansion} gives, for $b=h'(1)$,
\[
21.932<|b|<21.933.
\]
The phase condition gives $h'(-1)=ib$, and three integrations by parts in
\eqref{d3-Phi-inversion} give
\[
\Phi(z)=
\frac{b}{\pi z^2}\,\bigl(e^{iz/2}-ie^{-iz/2}\bigr)+R(z),
\quad
|R(z)|<
247.688\,\frac{e^{|\!\Im z|/2}}{|z|^3}.
\]
Interval arithmetic also gives
\[
\rho=\frac{2|b|^2}{\pi^2}>97.478.
\]
Then
\begin{equation}\label{d3-certified-cos}
\left|z^4\psi(z)-\rho \cos z\right|<
\frac{7120}{|z|}\,e^{|\!\Im z|},
\quad |z|\ge100\pi.
\end{equation}

For the finite part, interval arithmetic gives
\begin{equation}\label{d3-certified-signs}
\min_{2\le k\le100}
(-1)^k(k\pi)^4\psi(k\pi)>73.503.
\end{equation}
Hence $\psi$ has at least one zero in each interval
$(k\pi,(k+1)\pi)$, $2\le k<100$.

Let now $m\ge100$ and $|z|=m\pi$. On this circle,
\[
|\!\cos z|\ge\frac14\,e^{|\!\Im z|},
\]
which follows from
$|\!\cos(x+iy)|^2=\cos^2x+\sinh^2y$ and
$x^2+y^2=m^2\pi^2$. Since
\[
\frac{7120}{100\pi}<\frac{97.478}{4},
\]
\eqref{d3-certified-cos} and Rouch\'e's theorem show that
$z^4\psi(z)$ and $\rho \cos z$ have the same number of zeros in the disk
$|z|<m\pi$. The function $\cos z$ has $2m$ zeros in this disk, and
$\psi(0)=1$, so $\psi$ has exactly $2m-4$ zeros, counting multiplicities.

On the other hand, \eqref{d3-certified-cos} gives, for $k\ge100$,
\[
(-1)^k(k\pi)^4\psi(k\pi)>
97.478-\frac{7120}{100\pi}>0.
\]
Together with \eqref{d3-certified-signs}, this gives at least one positive
zero in each interval $(k\pi,(k+1)\pi)$, $k\ge2$. By the evenness of $\psi$,
there are at least $2m-4$ distinct real zeros in the disk $|z|<m\pi$.
Comparison with the count given by Rouch\'e's theorem shows that these are all
the zeros and that each of them is simple. Hence all zeros of $\psi$ are real
and simple.

By Subsection~\ref{subsubsec-d3-identification},
$\psi=\varphi_3$, $a=a_3$.
Thus the numerically found spectral branch is rigorously certified,
nondegenerate, and identified with the extremal branch.

\section{Final remarks}\label{sec-final}

This section collects several consequences of the main identities that are not
used in the proof of Theorem~\ref{thm-main}, together with some additional
interpretations of the formulas obtained above.

\subsection{An equation for the extremal function $\varphi$}

From \eqref{q-before-degree},
\[
q_d(z)=\frac{K_d(z)}{z^{d+1}}
-\frac{R_d(z)^2}{z^{2d+2}},
\]
so $q_d$ is an even Laurent polynomial.

Since $h_d$ is odd,
$\varphi(z)=\widetilde{\Phi}(z)\widetilde{\Phi}(-z)$. If $y_1$, $y_2$ satisfy
\[
y''+py'+qy=0,
\]
then, by direct differentiation, for $u=y_1y_2$ we obtain
\[
u'''+3pu''+(p'+4q+2p^2)u'
+(2q'+4pq)u=0.
\]
Applying this formula to \eqref{gauge-ODE} and the functions
$\widetilde{\Phi}(z)$, $\widetilde{\Phi}(-z)$, we get
\[
\varphi'''+\frac{3(d+1)}z\,\varphi''+
\Bigl(4q_d+\frac{(d+1)(2d+1)}{z^2}\Bigr)\varphi'+
\Bigl(2q_d'+\frac{4(d+1)}z\,q_d\Bigr)\varphi=0.
\]

\subsection{Finite part and zeta interpretation}\label{sec-fp-zeta}
Recall that in \eqref{Lambda-def} the coefficients $\Lambda_{2n+1}$ were
defined through the coefficients of the polynomial $P_d$, which represents
$2\widehat\sigma$ on the interval $(-1,1)$. We now show that
$\Lambda_{2n+1}$ also have an interpretation in terms of the zeros $\tau_k$ of
the extremal function $\varphi$: through the finite part of a meromorphic
series and through the zeta function of the zero sequence.

\subsubsection{Finite part}
Since $\Pi_d\equiv0$, \eqref{Pi-at-zero} gives
\[
\res_{z=\tau_k}\frac1{\varphi(z)}=
\frac1{\varphi'(\tau_k)} =(-1)^kV_d\mathcal C_d\,\tau_k^{d+1}.
\]
Define the finite part of the alternating series
\[
\sum_{k=1}^\infty(-1)^{k-1}\,\frac{\tau_k^d}{z^2-\tau_k^2}
\]
by
\begin{equation}\label{canonical-FP}
\operatorname{f.p.}\sum_{k=1}^\infty(-1)^{k-1}\,\frac{\tau_k^d}{z^2-\tau_k^2}=
z^{d-1}A(z)-Q_d(z).
\end{equation}
Then \eqref{phi-Theta} becomes
\[
\frac1{\varphi(z)}= 1-2V_d\mathcal C_d z^2
\operatorname{f.p.}\sum_{k=1}^\infty(-1)^{k-1}\,\frac{\tau_k^d}{z^2-\tau_k^2}.
\]
Definition \eqref{canonical-FP} does not require any additional summation
method: it is given by the already constructed functions $A$ and $Q_d$ and
does not assume convergence of the series
\[
\sum_{k\ge1}(-1)^{k-1}\tau_k^{2n+1}.
\]

\subsubsection{Zeta interpretation}
From the proof of the spectral gap (see Subsection~\ref{subsec-Theta-d}), for the upper
boundary value $A^+$ we have
\[
\widehat{A^+}(t)= -\frac{\pi}{4}\sum_{n=0}^m p_nt_+^{2n+1}+T(t), \quad \supp
T\subset[1,\infty),
\]
where $T$ is a distribution of finite order. By the Fourier--Laplace formula,
\[
A(iy)=\frac1{2\pi}\,\bigl\langle\widehat{A^+}(t),e^{-yt}\bigr\rangle,\quad
y>0.
\]
Therefore, as $y\to+\infty$,
\[
A(iy)= -\frac18\sum_{n=0}^m p_n(2n+1)!\,y^{-2n-2} +O(y^Me^{-y})
\]
for some $M>0$. By the definition of $\Lambda_{2n+1}$, this is equivalent to
\begin{equation}\label{A-asymptotic}
A(iy)= \sum_{n=0}^m \frac{\Lambda_{2n+1}}{(iy)^{2n+2}} +O(y^Me^{-y}).
\end{equation}
In particular, after subtracting the displayed terms, the remainder is
$O(y^{-N})$ for every~$N>0$.

For $\Re s>0$, set
\[
\eta_\tau(s)= \sum_{k=1}^\infty\frac{(-1)^{k-1}}{\tau_k^s}.
\]
Since $\tau_k\sim\pi k$, the series converges conditionally for
$0<\Re s\le1$ and absolutely for $\Re s>1$. Summation by parts also gives
uniform convergence on compact subsets of the half-plane $\Re s>0$, so
$\eta_\tau$ is holomorphic there.

In the strip $0<\Re w<1$,
\begin{equation}\label{eta-Mellin}
\eta_\tau(w)= -\frac2\pi\,\cos\frac{\pi w}{2} \int_0^\infty A(iy)y^{-w}\,dy.
\end{equation}
Indeed,
\[
A(iy)= -\sum_{k=1}^\infty(-1)^{k-1}\,\frac{\tau_k}{y^2+\tau_k^2},\quad
\int_0^\infty \frac{\tau}{y^2+\tau^2}\,y^{-w}\,dy= \frac{\pi}{2\cos(\pi
w/2)}\,\tau^{-w}.
\]
Termwise integration is justified by the estimate, uniform in $N$,
\[
\biggl|\sum_{k=1}^N(-1)^{k-1}\,\frac{\tau_k}{y^2+\tau_k^2}\biggr| \le
C\min(1,y^{-1}),
\]
which follows by summation by parts. For $0<\Re w<1$, the corresponding
majorant is integrable on $(0,\infty)$, so we may pass to the limit under the
integral sign.

The asymptotic formula \eqref{A-asymptotic} gives a meromorphic continuation of
the integral on the right-hand side of \eqref{eta-Mellin} to the left, and
therefore an analytic continuation of $\eta_\tau$. Indeed, successively
subtracting the asymptotic terms $\Lambda_{2n+1}/(iy)^{2n+2}$ at infinity
continues the right-hand side of \eqref{eta-Mellin} across the points
$w=-(2n+1)$. At each such point, the corresponding simple pole of the integral
is canceled by the zero of $\cos(\pi w/2)$, and
\[
\eta_\tau(-(2n+1))= \Lambda_{2n+1},\quad 0\le n\le m,
\]
where the left-hand side is understood by analytic continuation of the
function $\eta_\tau$. In particular,
\[
\eta_\tau(-d)=\Lambda_d=a_d.
\]

For $d=1$ there is only one coefficient,
$\Lambda_1=\eta_\tau(-1)$. A related interpretation in terms of the zeros of
the extremal function is given in \cite{Bo25}.

Thus the coefficients $\Lambda_{2n+1}$ give regularized values of the
alternating power sums
\[
\sum_{k\ge1}(-1)^{k-1}\tau_k^{2n+1}, \quad 0\le n\le m.
\]

\subsection{Even dimensions}
An essential part of the proof uses the oddness of the dimension. If
$d=2m+1$, then $|x|^{d-1}=x^{2m}$, and after the Fourier transform
multiplication by the weight $|x|^{d-1}$ becomes an ordinary derivative of
finite order. This is why the restriction of $\widehat\sigma$ to $(-1,1)$ is
a polynomial and can be canceled by the polynomial $Q_d$.

For even $d=2m$, the function $|x|^{d-1}=|x|^{2m-1}$ is no longer a
polynomial. On the Fourier side one obtains the nonlocal operator
\[
(-\partial_t^2)^{m-1/2}= (-1)^{m-1}\mathcal H\,\partial_t^{2m-1},
\]
where $\mathcal H$ is the Hilbert transform. Therefore, the restriction of
$\widehat\sigma$ to $(-1,1)$ need not be a polynomial.

A natural continuation of the present approach is to split the Hilbert
transform into the contribution from $(-1,1)$ and the exterior contribution,
and to study the resulting singular integral equation. Inversion formulas for
the finite Hilbert transform suggest looking for an analogue of $Q_d$ among
functions with algebraic singularities at $\pm1$, in particular with factors
$(1-t^2)^{\pm1/2}$. It remains to understand whether such a replacement leads
to a finite-dimensional spectral problem analogous to the odd-dimensional
case.

\subsection*{Acknowledgments}
The author is grateful to the AI model used in the preparation of this paper
for its assistance.


\begin{thebibliography}{99}

\bibitem{Ah79}
L.~V.~Ahlfors, \textit{Complex analysis},
3rd ed., McGraw--Hill, New York, 1979.

\bibitem{Bo25}
A.~Bondarenko, J.~Ortega-Cerd\`a, D.~Radchenko, K.~Seip, \textit{The
H\"ormander--Bernhardsson extremal function}, arXiv:2504.05205v2.

\bibitem{CL55}
E.~A.~Coddington, N.~Levinson, \textit{Theory of ordinary differential
equations}, McGraw--Hill, New York, 1955.

\bibitem{Da21}
F.~Dai, D.~Gorbachev, S.~Tikhonov, \textit{Estimates of the asymptotic
Nikolskii constants for spherical polynomials}, J.~Complexity \textbf{65}
(2021), 101553.

\bibitem{Ge64}
I.~M.~Gel'fand and G.~E.~Shilov, \textit{Generalized functions. Vol.~1:
Properties and operations}, Academic Press, New York--London, 1964.

\bibitem{Ge38}
J.~Geronimus, \textit{Sur un probl\`eme extr\'emal de Tchebycheff}, Izv. Akad.
Nauk SSSR Ser. Mat. \textbf{2}:4 (1938), 445--456.

\bibitem{Gu88}
G.~G.~Gundersen, \textit{Estimates for the logarithmic derivative of a
meromorphic function, plus similar estimates}, J. London Math. Soc. (2)
\textbf{37} (1988), 88--104.

\bibitem{HLP52}
G.~H.~Hardy, J.~E.~Littlewood, G.~P\'olya, \textit{Inequalities}, 2nd ed.,
Cambridge University Press, Cambridge, 1952.

\bibitem{Ho90}
L.~H\"ormander, \textit{The Analysis of Linear Partial Differential Operators
I: Distribution Theory and Fourier Analysis}, 2nd ed., Springer-Verlag, Berlin,
1990.

\bibitem{Ko98}
P.~Koosis, \textit{Introduction to $H^p$ Spaces}, 2nd ed., Cambridge Tracts in
Mathematics, Vol.~115, Cambridge University Press, Cambridge, 1998.

\bibitem{Le80}
B.~Ya.~Levin, \textit{Distribution of zeros of entire functions}, Providence,
RI, Amer. Math. Soc., 1980.

\bibitem{Ni75}
S.~M.~Nikolskii, \textit{Approximation of functions of several variables and
imbedding theorems}, Berlin; Heidelberg; New York: Springer, 1975.

\bibitem{Pe07}
I.~Pesenson, \textit{Plancherel--Polya-type inequalities for entire functions
of exponential type in $L^p(\mathbb R^d)$}, J. Math. Anal. Appl. \textbf{330}
(2007), 1194--1206.

\bibitem{Sh71}
H.~Shapiro, \textit{Topics in approximation theory}, Lecture Notes in
Mathematics, V.~187, Springer-Verlag, Berlin, Heidelberg, 1971.

\bibitem{Su25}
M.~Suzuki, \textit{Chains of reproducing kernel Hilbert spaces generated by unimodular
functions}, Ann. Inst. Fourier \textbf{75} (2025), no.~4, 1463--1508.

\bibitem{Se74}
G.~Szeg\"o, \textit{Orthogonal polynomials}, 3rd ed., American Mathematical
Society, 1974.

\bibitem{Ze05}
A.~Zettl, \textit{Sturm--Liouville theory}, Mathematical Surveys and
Monographs, Vol.~121, American Mathematical Society, Providence, RI, 2005.

\end{thebibliography}
\end{document}